\newcommand{\commentt}[2]{#1}
\newcommand{\comment}[1]{}

\newcommand{\cov}{{\rm Cov}}
\newcommand{\var}{{\rm Var}}
\newcommand{\tr}{{\rm tr}}
\newcommand{\std}{{\rm Std}}
\newcommand{\titt}{Unbiased time-average estimators for Markov chains}
\commentt{
\newcommand{\citet}{\citeasnoun}
\documentclass[a4paper,11pt]{article}
\usepackage{graphicx,amsthm}
\usepackage[full]{harvard}
\usepackage[ps,dvips]{xy}
\title{\titt}
\author{Nabil Kahal\'e
\thanks{\emph{ESCP Business School, 75011 Paris,
France; {e-mail: }{nkahale@escp.eu}.}}}
\date{\today}
\usepackage{amstext}
\usepackage{amssymb}
\usepackage{amsmath}
\usepackage{setspace}
\usepackage[margin=1in]{geometry}
\usepackage[english]{babel}
\bibliographystyle{dcu}
\usepackage{color}
\usepackage{textcomp}
\usepackage{varioref}
\usepackage{algpseudocode}
\usepackage{algorithm}

\usepackage{xspace}
\usepackage{tikz,pgfplots}
\usepackage{caption,subcaption}
\usepackage[normalem]{ulem}

\begin{document}

\newtheorem{example}{Example}[section]
\newtheorem{theorem}{Theorem}[section]
\newtheorem{conjecture}{Conjecture}[section]
\newtheorem{lemma}{Lemma}[section]
\newtheorem{proposition}{Proposition}[section]
\newtheorem{remark}{Remark}[section]
\newtheorem{corollary}{Corollary}[section]
\newtheorem{definition}{Definition}[section]
\numberwithin{equation}{section}
%\numberwithin{proposition}{section}
\maketitle
\newcommand{\ABSTRACT}[1]{\begin{abstract}#1\end{abstract}}
\newcommand{\citep}{\cite}
}
{
\documentclass[opre,nonblindrev]{informs3} 
\usepackage[margin=1in]{geometry}
\usepackage{textcomp}
\usepackage{pgfplots}
%\usepackage{hyperref}
%\OneAndAHalfSpacedXI % current default line spacing
%\OneAndAHalfSpacedXII
\DoubleSpacedXII
%%\DoubleSpacedXI

% If hyperref is used, dvi-to-ps driver of choice must be declared as
%   an additional option to the \documentclass. For example
%\documentclass[dvips,opre]{informs3}      % if dvips is used
%\documentclass[dvipsone,opre]{informs3}   % if dvipsone is used, etc.

% Private macros here (check that there is no clash with the style)

% Natbib setup for author-year style
\usepackage[round]{natbib}
 \bibpunct[, ]{(}{)}{,}{a}{}{,}%
 \def\bibfont{\small}%
 \def\bibsep{\smallskipamount}%
 \def\bibhang{24pt}%
 \def\newblock{\ }%
 \def\BIBand{and}%

%% Setup of theorem styles. Outcomment only one.
%% Preferred default is the first option.
\TheoremsNumberedThrough     % Preferred (Theorem 1, Lemma 1, Theorem 2)
%\TheoremsNumberedByChapter  % (Theorem 1.1, Lema 1.1, Theorem 1.2)
\ECRepeatTheorems

%% Setup of the equation numbering system. Outcomment only one.
%% Preferred default is the first option.
\EquationsNumberedThrough    % Default: (1), (2), ...
%\EquationsNumberedBySection % (1.1), (1.2), ...

% In the reviewing and copyediting stage enter the manuscript number.
\MANUSCRIPTNO{} % When the article is logged in and DOI assigned to it,
                 %   this manuscript number is no longer necessary

%%%%%%%%%%%%%%%%
\renewcommand{\qed}{\halmos}

\begin{document}
\RUNAUTHOR{Kahal\'e}
\RUNTITLE{Multilevel  methods for discrete Asian options}
\TITLE{General multilevel Monte Carlo methods for pricing discretely monitored
Asian options}
\ARTICLEAUTHORS{%
\AUTHOR{Nabil Kahal\'e}

\AFF{ESCP Business School, 75011 Paris, France, \EMAIL{nkahale@escp.eu} \URL{}}
} % end of the block
\KEYWORDS{
discretely monitored Asian option, multilevel Monte Carlo method, option pricing, variance reduction}
\bibliographystyle{informs2014} 
}

\ABSTRACT{
We consider a time-average estimator \(f_{k}\) of a functional of a Markov chain. Under a coupling assumption, we show that the expectation of \(f_{k}\) has a limit \(\mu\) as the number of time-steps goes to infinity. We describe a modification of \(f_{k}\)  that yields an unbiased estimator  \(\hat f_{k}\)  of \(\mu\). It is shown that  \(\hat f_{k}\)   is square-integrable  and has finite expected running time. Under certain conditions,    \(\hat f_{k}\)  can be built  without any precomputations, and is  asymptotically at least as efficient as  \(f_{k}\), up to a multiplicative constant  arbitrarily close to \(1\). Our approach  provides an unbiased  estimator for the bias of \(f_{k}\). We study applications to  volatility forecasting,  queues, and  the simulation of high-dimensional Gaussian vectors. Our numerical experiments are consistent with our theoretical findings. 
     }
     
\commentt{
Keywords:  multilevel Monte Carlo, unbiased estimator, steady-state, Markov chain, time-average estimator  
}{\maketitle
Subject classifications: Finance: asset pricing; Simulation: Efficiency ; Analysis of algorithms: Computational complexity}
\section{Introduction}
Markov chains arise in a variety of fields such as queuing  networks,  machine learning and health-care. The steady-state  of certain Markov chains   is accurately determined via analytical tools. For instance, in a \(M/M/m\) queue with \(m\) servers and exponentially distributed interarrival  and service times, the steady-state  distribution of customers in the system  is given by a simple analytical formulae. On the other hand, the steady-state behavior of  queuing networks with generally distributed interarrival and service times is intractable (see \cite{BandiBertsimasYoussef2015} for a detailed discussion).  Monte Carlo simulation can be used to study the steady-state of intractable systems. In general,     Monte Carlo simulation has a high computation cost, but its performance can be improved via  variance reduction techniques such as the control variate technique, moment matching, stratified sampling, importance sampling~\cite{glasserman2004Monte,asmussenGlynn2007} and multilevel Monte Carlo (MLMC) \cite{giles2015multilevel}. The related  Quasi-Monte Carlo method often outperforms standard Monte Carlo simulation in low-dimensional problems and in pricing of financial derivatives~\cite{glasserman2004Monte}.     
Another issue with Monte Carlo simulation is that it sometimes produces biased estimators. For instance, the price of a financial derivative obtained by standard Monte Carlo simulation and  discretization of a stochastic differential equation  is usually biased. Randomized Multilevel Monte Carlo methods (RMLMC) that provide unbiased estimators for expectations of functionals  associated with    stochastic differential equations are given in
\cite{mcleish2011,GlynnRhee2015unbiased}.
\comment{ Algorithms that find an optimal randomization distribution for RMLMC  estimators  are described in \cite{GlynnRhee2015unbiased,cui2021optimal,kahale2019optimal}.}\citet{jacob2015nonnegative} study the existence of unbiased nonnegative estimators. Unbiased estimators   have been used in  diverse settings including   Markov chain Monte Carlo methods~\cite{bardenet2017,agapiou2018unbiased,middleton2018unbiased,jacob2020unbiased},  estimating the expected cumulative discounted cost~\cite{cui2020optimalDisc},  pricing of discretely monitored Asian options~\cite{kahale2020Asian},  inference for hidden Markov model diffusions~\cite{vihola2021unbiasedInference}, and estimating the gradient of the  log-likelihood~\cite{jasra2021unbiased}. \citet{Vihola2018} describes  stratified versions of  RMLMC methods that, under general conditions, are  asymptotically as
efficient as MLMC. Unbiased estimators have the following advantages. First, a confidence interval is easily calculated from independent replications of an unbiased estimator. Second, taking the average of \(m\) independent copies of an unbiased estimator produces an unbiased estimator with a variance equal to that of the original estimator divided by \(m\). This leads to an efficient parallel computation of an unbiased estimator.

This paper considers a Markov chain  \((X_i,i\geq 0)\)  with state-space
\(F\) and deterministic initial value \(X_{0}\).
Let \(f\) be a  deterministic real-valued
measurable function on \(F\) such that \(f(X_{i})\) is square-integrable for \(i\geq0\). For \(k\ge1\), define the time-average estimator  
\begin{equation*}
f_{k}:=\frac{1}{k-b(k)}\sum ^{k-1}_{i=b(k)}f(X_{i}),
\end{equation*}
where  \(b(k)\ge0\) is a burn-in period that may depend on \(k\). The estimator \(f_{k}\) is often used to estimate the limit \(\mu\) of \(E(f(X_{m}))\) as   \(m\) goes to infinity, when such a limit exists.   \citet{whitt1991longRun} studies the performance of time-average estimators in a continuous-time framework. He  provides evidence that, in general, one long time-average estimator is more efficient than several independent replications of time-average estimators of shorter length. He  finds that, if the simulation length is large enough to obtain reasonable estimates of \(\mu\), then several independent replications are almost as efficient as one longer run. He also shows that, in general,  it is not efficient to run a very large number of independent replications with very short length.  

Time-average estimators have been used in various contexts such as  the sampling from a posterior distribution \cite{tierney1994markov}, computing the volume of a convex body \cite{VempalaCousins2016}, and estimating the steady-state performance metrics of time-dependent queues   \cite{whitt2019time}. For general Markov chains, however, time-average estimators have the following drawbacks~\cite[p. 96]{asmussenGlynn2007}. First, they are usually biased 
because, in general, the distribution of the   \(X_{i}\)'s is not the steady-state distribution. Second, because of the bias and since  the \(f(X_{i})\)'s are usually correlated,  calculating  confidence intervals  from time-average estimators is challenging.   The method of batch means (BM) divides a single time-average into several consecutive batches, and calculates an asymptotic confidence interval from the averages over each batch.  The quality of this confidence interval depends on the extent to which these averages are independent, identically distributed and Gaussian~\cite[p. 110]{asmussenGlynn2007}. A confidence interval can also be calculated via the method of independent replications (IR), that simulates independent copies of \(f_{k}\), but the quality of this confidence interval depends on the bias \(E(f_{k})-\mu\)  of \(f_{k}\). \citet{argon2013steady} study variants of the BM and IR methods.

Following 
 \citet{whitt1991longRun} and assuming \(\mu\neq0\), the bias of \(f_{k}\) can be reduced by setting \(b(k)\) equal to the smallest integer \(s\) such that \(|E(f(X_{i}))-\mu|\leq |\mu|\epsilon\) for \(i\ge s\),  where \(\epsilon\) is a small constant such as \(0.01\) or \(0.001\) (see also   \cite[p. 102]{asmussenGlynn2007}). In other words, the relative absolute bias is at most \(\epsilon\) at any time-step \(i\) larger than or equal to \(s\).  Such \(s\)  is closely related to the relaxation time (\citet[Chap. IV]{asmussenGlynn2007}), and an analytic expression or approximation for \(s\) or for the bias has been calculated  for certain Markov chain functionals. For instance,  \citet{whitt1991longRun}
 calculates \(s\) analytically  for the number of busy servers in an \(M/G/\infty\) queue, and provides an analytic approximation for \(s\) for the    \(M/M/1\) queue length process.   \citet[Chap. IV]{asmussenGlynn2007} show that, for general Markov chains and under suitable conditions, the bias of \(f_{k}\) is of order \(1/k\)  when \(b(k)=0\), and give an analytic approximation for the bias at a given time  for the \(GI/G/1\) queue waiting time process.  While explicit convergence rates  to the steady-state distribution have been established in the previous literature for many Markov chains (e.g. \cite{diaconis1991geometric,sinclair1992improved,VempalaCousins2016,kahaleGaussian2019,barkhagen2021stochastic,sinclair2022critical}), the mixing time of other Markov chains that arise in practice is not  formally known~\cite{diaconis2009markov}.\  Furthermore, the dependence  of the bias  \(E(f(X_{i}))-\mu\) on  \(i\) does not always follow the same pattern: in the  \(M/G/\infty\) queue example,    the  bias can decay polynomially or exponentially in \(i\), depending on the service time distribution. In the absence of knowledge on the relaxation time,    \citet[p. 102]{asmussenGlynn2007} suggest to select \(b(k)\) in an ad-hoc manner, by setting    \(b(k)=\lfloor k/10\rfloor\) for instance. 

The previous discussion shows that the  bias makes it  difficult to ascertain the quality of time-average estimators for general Markov chains. This paper provides a randomized multilevel framework for estimating and correcting the bias in time-average estimators. Under suitable conditions, we first construct a RMLMC unbiased estimator of the bias of a time-average estimator. Combining this estimator with a conventional time-average estimator yields  an unbiased estimator  \(\hat f_{k}\) of \(\mu\), that is, \(E(\hat f_{k})=\mu\).     Our construction is based on a coupling assumption and a  time-reversal transformation inspired from \citet{glynn2014exact}, and  a RMLMC estimator introduced by \citet{GlynnRhee2015unbiased}.  A similar  coupling is used in \cite{kahaRandomizedDimensionReduction20,KAHALE2022}
to design and analyse variance reduction algorithms for
time-varying
Markov chains with finite horizon. The main contributions of our paper are as follows:
\begin{enumerate}
\item Our approach constructs an unbiased square-integrable estimator, that can be simulated in finite expected time, of the bias of a time-average estimator. This  allows to estimate the bias and to determine the number of time-steps needed to substantially reduce it.
\item   \(\hat f_{k}\) is an unbiased estimator of \(\mu\), is square-integrable and can be computed in finite expected time. For a suitable choice of parameters and under certain assumptions, the work-normalized variance of   \(\hat f_{k}\) is at most equal to that of    \(f_{k}\), up to a multiplicative factor that can be made arbitrarily close to \(1\) as \(k\) goes to infinity. As shown by~\citet{glynn1992asymptotic}, the efficiency of an unbiased estimator can be measured through the work-normalized variance, i.e., the product of the variance and expected running time.     The smaller the work-normalized variance, the higher the efficiency. The performance of a biased estimator such as   \(f_{k}\)    incorporates its bias, in addition to its  variance \cite[p. 16]{glasserman2004Monte}. Thus, for an appropriate choice of parameters, \(\hat f_{k}\) is at least as efficient as  \(f_{k}\) as \(k\) goes to infinity, up to a multiplicative factor arbitrarily close to \(1\). 
\item Under suitable conditions, \(\hat f_{k}\) can be constructed without any precomputations or  knowledge of  the relaxation time or related properties of the chain.  Furthermore, our approach does not require any recurrence properties of the chain. In our numerical experiments, that use a conservative choice for the parameters,   \(f_{k}\) is about twice as efficient as    \(\hat f_{k}\) for large values of \(k\).
  
\comment{\item 
By simulating \(m\) independent replicates of   \(\hat f_{k}\), the variance of \(\hat f_{k}\)  can be estimated and a confidence interval for \(\mu\) can be constructed.
\item The average of the \(m\) replicates yields an unbiased estimator  \(\hat f_{k,m}\)    of \(\mu\) with a work-normalized variance at most equal to that of \(f_{k}\), up an additive term that goes to \(0\) as \(k\) goes to infinity. Thus, the estimation of \(\mu\) can be parallelized with the same asymptotic efficiency. 
}\end{enumerate}
For general Markov chains, we are not aware of a previous construction of an efficient   unbiased estimator of the bias of a time-average estimator, or of efficient unbiased estimators for \(\mu\) based on time-averaging. Assuming that \(f\) is Lipschitz and that \(X\) is `contractive on average',   \citet{glynn2014exact} construct square-integrable unbiased  RMLMC estimators for the steady-state expectation of Markov chain functionals. In view of the time-reversal transformation and  RMLMC estimator used, our techniques are  closely related to theirs. However, their method is not based on time-averaging, and our approach does not require  \(f\) to be Lipschitz or \(X\) to be contractive on average. We provide several examples where \(f\) is discontinuous and our method is provably efficient.           \citet{glynn2014exact}  also describe another unbiased estimator for positive recurrent
Harris chains. \citet{jacob2020unbiased} study unbiased Markov Chain Monte Carlo methods that use time-averaging.

The rest of the paper is organised as follows. Section~\ref{se:conventionalEstimators} presents the coupling assumption and  studies conventional time-average estimators under this assumption. In particular, it shows that the mean square error  \(E((f_{k}-\mu)^{2})\) is of order \(1/k\).
Section~\ref{se:UnbiasedEstimators} 
describes and analyses  an unbiased estimator of the bias of a time-average estimator. It also constructs  and studies   \(\hat f_{k}\) as well as a stratified version of   \(\hat f_{k}\)   under the coupling assumption. Section~\ref{se:examples} provides examples and Section~\ref{se:NumerExper} presents numerical experiments.
Omitted proofs are in the appendix.
Throughout the  paper, the running time refers to the number of arithmetic operations. For simplicity, it is supposed  that \( b(k)\le k/2\) for \(k\geq1\), and that the expected time to simulate  \(f_{k}\) is   \(k\) units of time. We assume   that  there are independent  and identically distributed (i.i.d.) random variables \(U_{i}\),  \(i\ge0\), that take values in a measurable space \(F'\),    and a measurable
function \(g\) from \(F\times F'\) to \(F\) such
that, for \( i\geq 0\), \begin{equation}\label{eq:recXi}
X_{i+1}=g(X_{i},U_{i}).
\end{equation} 
\section{Conventional time-average estimators}\label{se:conventionalEstimators} 
We introduce the coupling assumption in Subsection~\ref{sub:coupling} and use it in Subsection~\ref{sub:conventionalLongRun} to  establish bounds  on the bias, standard deviation and mean square error of conventional time-average estimators. Subsection~\ref{sub:Sharp}  describes an example showing the sharpness of the standard deviation and mean square error bounds.        

 \subsection{The coupling assumption}\label{sub:coupling}
Extend the  random sequence \((U_{i}, i\ge0)\) to all   \(i\in\mathbb{Z}\), so that  \(U_{i}\), \(i\in\mathbb{Z}\),  are i.i.d. random variables taking values in \(F'\).  For \(i\geq0\),  define recursively the measurable function  \(G_{i}\)  from \({ F\times F'}^{i}\) to \(F\)  by setting \(G_{0}(x):=x\) and\begin{equation*}
G_{i+1}(x;u_{0},\dots,u_{i}):=g(G_{i}(x;u_{0},\dots,u_{i-1}),u_{i}),
\end{equation*}for \(x\in F\) and \(u_{0},\dots,u_{i}\in F'\). It can be shown by  induction that,  for \(i\geq0\), \begin{equation}\label{eq:xiGi}
X_{i}=G_{i}(X_{0};U_{0},\dots,U_{i-1}).
\end{equation}  For \(m\in\mathbb{Z}\) and \(i\geq-m\), let
\begin{equation}\label{XimDefinition}
X_{i,m}:=G_{i+m}(X_{0};U_{-m},U_{-m+1},\dots,U_{i-1}).
\end{equation}   
Thus     \(X_{i,0}=X_{i}\) for \(i\geq0\). By \eqref{eq:xiGi},    \(X_{i,m}\sim X_{i+m}\) for \(m\in\mathbb{Z}\) and \(i\geq-m\), where `\(\sim\)' denotes equality in distribution. Furthermore,\begin{equation}\label{eq:recXim}
X_{-m,m}=X_{0},\text{ and } X_{i+1,m}=g(X_{i,m},U_{i}).
\end{equation} In other words,  \((X_{i,m}, i\geq -m)\) is a Markov chain that is a copy of    \((X_{i}, i\geq 0)\),  and is driven by  \((U_{i}, i\geq -m)\).  For \(i,m\geq0\), the last \(i\) random variables  driving the calculation of \(X_{i}\) and \(X_{i,m}\), i.e., \(U_{0},\dots,U_{i-1}\), are the same. This leads us to state the following assumption.
\begin{description}
\item[Assumption 1 (A1).]    There is a positive decreasing sequence \((\nu(i),i\geq0)\)  such that
\begin{equation}\label{eq:infiniteSumAssumption}
\sum^{\infty}_{i=0}\sqrt{\frac{\nu(i)}{i+1}}<\infty,
\end{equation}and, for \(i,m\geq0\),\begin{equation}\label{eq:rhoDef}
E((f(X_{i,m})-f(X_{i}))^{2})\leq\nu(i).
\end{equation}
\end{description}
Intuitively speaking, \eqref{eq:rhoDef} holds with a small \(\nu(i)\) if, for \(h\geq i\),
\(f(X_{h})\) is mainly determined by \((U_{h-i},\dots,U_{h-1})\), that is, if \(f(X_{h})\) depends to a large extent on the last \(i\) copies of \(U_{0}\) driving the Markov chain \((X_{k}:0\leq k\le h)\).  
Proposition~\ref{pr:relax} shows that Assumption A1 holds under a condition similar to  \eqref{eq:rhoDef}. Note that   \eqref{eq:rhoDef} and  \eqref{eq:rhoDefx} are identical if \(X_{0}=x\). \begin{proposition}\label{pr:relax}
Suppose there is \(x\in F\) and a positive decreasing sequence \((\nu'(i),i\geq0)\)   that satisfies \eqref{eq:infiniteSumAssumption} and, for \(i,m\geq0\),\begin{equation}\label{eq:rhoDefx}
E((f(G_{i}(x;U_{0},\dots,U_{i-1}))-f(X_{i,m}))^{2})\leq\nu'(i).
\end{equation}
Then Assumption A1 holds with \(\nu(i)=4\nu'(i)\) for \(i\geq0\). 
\end{proposition}
\begin{proof}
Applying \eqref{eq:rhoDefx} with \(m=0\) shows that, for \(i\geq0\), \begin{equation*}
E((f(G_{i}(x;U_{0},\dots,U_{i-1}))-f(X_{i}))^{2})\leq\nu'(i).
\end{equation*} 
Together with  \eqref{eq:rhoDefx}, and since \(E((Z+Z')^{2})\le2(E(Z^{2})+E({Z'}^{2}))\) for square-integrable random variables \(Z\) and \(Z'\), this implies that, for \(i,m\geq0\),
\begin{equation*}
E((f(X_{i,m})-f(X_{i}))^{2})\leq4\nu'(i).
\end{equation*}
\end{proof}  

Assumption A2 stated below is stronger than Assumption A1 and  says that the \(\nu(i)\)'s decay exponentially with \(i\).
\begin{description}
\item[Assumption 2 (A2).] Assumption A1 holds with \(\nu(i)\le c e^{-\xi i}\) for \(i\geq0\), where  \(c\) and \(\xi\) are positive constants with \(\xi\leq1\).
 \end{description}Proposition~\ref{pr:contractive1} shows that Assumption A1 holds under   certain conditions. As \(\eta\leq e^{\eta-1}\) for \(\eta\in\mathbb{R}\), Assumption A2 holds as well under the same conditions.    
\begin{proposition}
\label{pr:contractive1}Assume that \(F\) is a metric space with metric \(\rho:F\times F\rightarrow\mathbb{R}_{+}\) and there   are positive constants \(\eta\), \(\kappa\), \(\kappa'\) and \(\gamma\) with \(\eta<1\) such that, for \(i,m\geq0\), \begin{equation}\label{eq:defkappaLips}
E((f(X_{i,m})-f(X_{i}))^{2})\leq \kappa^{2}( E(\rho^{2}(X_{i,m},X_{i})))^{\gamma}, \end{equation} 
and \begin{equation}\label{eq:distanceFinite}
E(\rho^{2}(X_{0},X_{i}))\leq\kappa',
\end{equation}and, for \(x,x'\in F\),
\begin{equation}\label{eq:distanceContractive}
E(\rho^{2}( g(x,U_{0}),g(x',U_{0})))\leq\eta \rho^{2}( x,x').
\end{equation} 
Then Assumption A1 holds with \(\nu(i)=\kappa^{2}\kappa'^{\gamma}\eta^{\gamma i}\) for \(i\geq0\).\end{proposition}
The generalized Lipschitz condition~\eqref{eq:defkappaLips} obviously holds for Lipschitz functions. Examples of non-Lipschitz functions, including discontinuous functions, that satisfy~\eqref{eq:defkappaLips}, are given by  \citet{kahaleGaussian2019}  in the context of simulating high-dimensional Gaussian vectors.  Condition~\eqref{eq:distanceFinite} says that the expected square distance between \(X_{0}\) and \(X_{i}\) is bounded. The contractivity  condition~\eqref{eq:distanceContractive} is used by~\citet{glynn2014exact} to obtain unbiased estimators for Markov chains.
\subsection{Convergence properties}\label{sub:conventionalLongRun}
Given a non-negative sequence \((\omega(i),i\geq0)\) such that \(\sum^{\infty}_{i=0}\sqrt{{\omega(i)}/(i+1)}\) is finite,  set
\begin{equation*}
\overline{\omega}(j):=\sum^{\infty}_{i=j}\sqrt{\frac{\omega(i)}{i+1}},
\end{equation*}   for \(j\ge0\). Note that \(\overline{\omega}(j)\)  is finite and is a decreasing function of \(j\), and that  \(\overline{\omega}(j)\)   goes to \(0\) as \(j\) goes to infinity. 
Theorem \ref{th:biasedLongRun} shows that, under Assumption A1, the sequence  \((E(f(X_{h})),h\ge0)\)  is convergent and examines the convergence properties of \(f(X_{h})\) and of  standard  time-average estimators.  
\begin{theorem}
\label{th:biasedLongRun}Suppose that Assumption A1 holds. Then \(E(f(X_{h}))\) has a finite limit \(\mu\)  as \(h\) goes to infinity. For \(h\ge 0\), \begin{equation}\label{eq:thBias}
|E(f(X_{h}))-\mu|\leq\sqrt{\nu(h)}.
\end{equation}For  \(h\ge 0\) and  \(k>0\),
 \begin{equation}\label{eq:thLongRunBias}
|E(\frac{1}{k}\sum^{h+k-1}_{i=h}f(X_{i}))-\mu|\leq \frac{\overline{\nu}(\lfloor h/2\rfloor)}{\sqrt{k}},
\end{equation}
and\begin{equation}\label{eq:thMSE}
E\left(\left((\frac{1}{k}\sum^{h+k-1}_{i=h}f(X_{i}))-\mu\right)^{2}\right)\leq\frac{26(\overline{\nu}(0))^{2}}{k}.
\end{equation}
\end{theorem}
Equations \eqref{eq:thBias},  \eqref{eq:thLongRunBias} and \eqref{eq:thMSE}  provide upper-bounds  on the absolute bias of \(f(X_{h})\),  and on the absolute  bias and  mean square error  of time-average estimators of \(\mu\).
Under Assumption A1, Theorem~\ref{th:biasedLongRun} implies that \(f_{k}\) is an estimator of \(\mu\) with mean square error \(E((f_{k}-\mu)^{2})=O(1/k)\). Furthermore, if \(b(k)\)  goes to infinity with \(k\), then  \(|E(f_{k})-\mu|=o(1/\sqrt{k})\), i.e.,   \(\sqrt{k}|E(f_{k})-\mu|\) goes to \(0\) as \(k\) goes to infinity.
Also, if  \(\sum^{\infty}_{i=0}\sqrt{{\nu(i)}}<\infty\), then \eqref{eq:thBias} implies immediately a bound of order \(1/k\) on \(|E(f_{k})-\mu|\). In both cases, the absolute bias  \(|E(f_{k})-\mu|\) is asymptotically negligible, as \(k\) goes to infinity, in comparison with the bound of  order \(1/\sqrt{k}\)   on \(\std(f_{k})\)  implied by Lemma \ref{le:StdSum} below. 
The proof of Theorem~\ref{th:biasedLongRun} relies on Lemma \ref{le:StdSum}.
\begin{lemma}\label{le:StdSum}
Suppose that A1 holds. Then, for \(h\ge0\) and  \(k>0\),\begin{equation}\label{eq:stdSumBound}
\std(\frac{1}{k}\sum^{h+k-1}_{i=h}f(X_{i}))\leq\frac{5\overline{\nu}(0)}{\sqrt{k}}.
\end{equation}
\end{lemma}
The mean square error bound \eqref{eq:thMSE} is proportional to \((\overline{\nu}(0))^{2}\). Proposition~\ref{pr:UpperBoundDelta} provides  bounds on \(\overline{\nu}(0)\). Under Assumption A2, the bound on \(\overline{\nu}(0)\) is inversely proportional to \(\sqrt{\xi}\). Under an additional decay assumption on \(\nu\), it is a polylogarithmic function of \(\xi\).
  
\begin{proposition}\label{pr:UpperBoundDelta}Suppose that Assumption A2 holds. Then 
\begin{equation}\label{eq:deltaBound}\overline{\nu}(0)\leq9\sqrt{\frac{ c}{\xi}}.\end{equation}
Moreover, if \(\nu(i)\leq c/(i+1)\) for \(i\geq0\) then\begin{equation}\label{eq:deltaBoundGeom}
\overline{\nu}(0)\leq14\sqrt{c}\ln\left(\frac{2}{\xi}\right).
\end{equation}
\end{proposition}
\subsection{Sharpness of bounds}\label{sub:Sharp}
This subsection gives an example proving  the optimality of   \eqref{eq:thMSE},  \eqref{eq:stdSumBound}  and \eqref{eq:deltaBound}, up to a multiplicative constant.
Consider the real-valued autoregressive sequence   \((X_{i},i\geq0)\) given by the recursion \begin{equation*}
X_{i+1}=\sqrt{\eta}X_{i}+U_{i},
\end{equation*}    for \( i\geq 0\), with \(X_{0}=0\), where \(\eta\in[0,1)\) and    \(U_{i},i\geq0\),   are real-valued i.i.d. with \(E(U_{i})=0\) and \(\var(U_{i})=1\). In this example, \(F=F'=\mathbb{R}\) and \(g(x,u)=\sqrt{\eta}x+u\).
Assume that \(f\) is the identity function on \(\mathbb{R}\) and  that \(b(k)=0\). It is easy to verify by induction that \(E(X_{i})=0\) and \(\var(X_{i})\leq1/(1-\eta)\) for \(i\geq0\).  The conditions in Proposition~\ref{pr:contractive1} hold for the Euclidean distance \(\rho(x,x')=|x-x'|\) for \((x,x')\in\mathbb{R}^{2}\), with \(\kappa=\gamma=1\) and \(\kappa'=1/(1-\eta)\). Thus,  Assumption A1 holds with \(\nu(i)=\eta^{ i}/(1-\eta)\)  for \(i\geq0\), and Assumption A2 holds with  \(c=1/(1-\eta)\) and \(\xi=1-\eta\). Applying  \eqref{eq:thMSE} with \(h=0\)  and noting that \(\mu=0\) shows, in combination with \eqref{eq:deltaBound}, that\begin{displaymath}
 \var(f_{k})\leq  \frac{2106}{k(1-\eta)^{2}}.
\end{displaymath}
On the other hand, it can be shown by induction that, for \(i\geq0\), \begin{displaymath}
X_{i}=\sum^{i-1}_{j=0}(\sqrt{\eta})^{i-1-j}U_{j},
\end{displaymath}
and, for \(k\geq0\),\begin{displaymath}
\sum_{i=0}^{k}X_{i}=\sum^{k}_{j=0}\frac{1-(\sqrt{\eta})^{k-j}}{1-\sqrt{\eta}}U_{j}.
\end{displaymath}  
Consequently, \begin{displaymath}
\var(\sum_{i=0}^{k}X_{i})=\sum^{k}_{j=0}{\alpha _{j}}^{2},
\end{displaymath} where  \(\alpha _{j}:=(1-(\sqrt{\eta})^{j})/(1-\sqrt{\eta})\). By standard calculations,  \(2\alpha _{j}\ge1/(1-\eta)\) for \(j\geq j_{0}\), where \(j_{0}:=\lceil2/\log_{2}(1/\eta)\rceil\). Thus, for  \(k\geq2j_{0}\), we have \begin{displaymath}
\var(f_{k})\geq\frac{1}{8k(1-\eta)^{2}}.
\end{displaymath}
This implies that  \eqref{eq:thMSE} as well as \eqref{eq:deltaBound} are tight, up to an absolute multiplicative constant. The same calculations show that \eqref{eq:stdSumBound} is tight, as well.
\section{Unbiased time-average estimators}\label{se:UnbiasedEstimators}
Subsection \ref{sub:singleTermEstimator} recalls the \emph{single term estimator},  a RMLMC estimator      introduced by \citet{GlynnRhee2015unbiased}. Subsection \ref{sub:unbiasedLongRun}  uses this estimator to construct an unbiased time-average estimator \(\hat f_{k}\). Subsection~\ref{sub:Sdependent} shows how to choose the parameters used to construct    \(\hat f_{k}\) in order to ensure that   \(\hat f_{k}\) has good convergence properties. Some of these parameters are calculated in terms  of  \(\bar\nu\), though.
Under additional assumptions, Subsection~\ref{sub:Oblivious} provides choices for these parameters without explicit knowledge of  \(\bar\nu\). Subsection \ref{sub:strat} describes a stratified version of   \(\hat f_{k}\). Subsection~\ref{sub:implementationDetails} gives implementation details.      
\subsection{The single term estimator}\label{sub:singleTermEstimator} Let   \((Y_{l}, l\geq0)\) be a sequence of   square-integrable random variables such that \(E(Y_{l})\) has a limit \(\mu_{Y}\) as \(l\) goes to infinity.    Consider a probability distribution \((p_{l},l\geq0)\)  such that  \(p_{l}>0\) for \(l\geq0\). Let   \(N\in\mathbb{N}\)  be an integral  random variable independent of  \((Y_{l}, l\geq0)\)  such that \(\Pr(N=l)=p_{l}\)   for \(l\geq0\). Theorem~\ref{th:Glynn}, due to \citet{GlynnRhee2015unbiased}  (see also~\cite[Theorem 2]{Vihola2018}),  describes the single term estimator \(Z\) and shows that, under suitable conditions, it has expectation equal to  \(\mu_{Y}\).  
\begin{theorem}[\cite{GlynnRhee2015unbiased}]\label{th:Glynn}    Set \(Z:=(Y_{N}-Y_{N-1})/p_{N}\), with  \(Y_{-1}:=0\). If \(\sum^{\infty}_{l=0}E((Y_{l}-Y_{l-1})^{2})/p_{l}\) is finite  then \(Z\) is square-integrable, \(E(Z)=\mu_{Y}\), and
\begin{equation*}
E(Z^{2})=\sum^{\infty}_{l=0}\frac{E((Y_{l}-Y_{l-1})^{2})}{p_{l}}.
\end{equation*}
\end{theorem}
\subsection{Construction  of    \(\hat f_{k}\)  }\label{sub:unbiasedLongRun}
This subsection supposes that Assumption A1 holds and constructs     \(\hat f_{k}\) along the following steps:\begin{enumerate}
\item 
Build a random sequence \((f_{k,l},l\geq0)\) such that \(E(f_{k,l})\rightarrow\mu\) as \(l\) goes to infinity and  \(f_{k,0}\) is a standard time-average estimator with burn-in period   \(b'(k)\in[b(k),k/2]\).
\item Use the sequence  \((f_{k,l},l\geq0)\) to construct a RMLMC estimator \(Z_{k}\) with \(E(Z_{k})=\mu-E(f_{k,0})\).
\item Combine \(f_{k,0}\) and  \(Z_{k}\)  to produce  \(\hat f_{k}\).
\end{enumerate}

First, we detail Step 1. For  \(k\ge1\), let \(b'(k)\) be a burn-in period with \(b(k)\le b'(k)\le k/2\).   Different choices for \(b'(k)\) will be studied in Subsections~\ref{sub:Sdependent} and~\ref{sub:Oblivious}. For  \(k\ge1\) and \(l\ge0\), let \begin{equation}\label{eq:fklDef}
f_{k,l}:=\frac{1}{k-b'(k)}\sum ^{k-1}_{i=b'(k)}f(X_{i,k(2^{l}-1)}).
\end{equation}
In particular,
\begin{equation}\label{eq:fk0def}
f_{k,0}=\frac{1}{k-b'(k)}\sum ^{k-1}_{i=b'(k)}f(X_{i}).
\end{equation}As \(X_{i,k(2^{l}-1)}\sim X_{i+k(2^{l}-1)}\), Theorem~\ref{th:biasedLongRun} implies that \(E(f_{k,l})\rightarrow\mu\) as \(l\) goes to infinity. By \eqref{XimDefinition}, for \(0\leq l<l'\) and \(i\geq0\), the last \(i+m\) copies of \(U_{0}\) used to calculate  \(f_{k,l}\) and  \(f_{k,l'}\) are the same, where \(m=k(2^{l}-1)\). Thus, intuitively speaking,  \(f_{k,l'}\) should be close to     \(f_{k,l}\)  for large values of \(l\), and increasing  \(b'(k)\) should make  \(f_{k,l'}\) closer to     \(f_{k,l}\) even for small \(l\). 
For simplicity, it is assumed that the expected time to simulate \(f_{k,l}\) is equal to  \(k2^{l}\). This assumption is justified by the fact that \(f_{k,l}\) is calculated by generating \(U_{-m},\dots,U_{k-2}\),  
and using \eqref{eq:recXim} to calculate \(X_{-m,m},\dots,X_{k-1,m}\). Lemma~\ref{le:boundVarfk0} gives an upper bound on the  variance of \(f_{k,0}\) in terms of that of \(f_{k}\).  \begin{lemma}\label{le:boundVarfk0}For  \(k\ge1\), 
\begin{equation*}
\var(f_{k,0})
\le\frac{796(\overline{\nu}(0))^{2}}{k^{3/2}}\sqrt{b'(k)-b(k)}+\var(f_{k}).\end{equation*}
\end{lemma}
Next, we describe Step 2. Let \((p_{l},l\geq0)\) be a probability distribution on \(\mathbb{N}\) with \(p_{l}>0\) for \(l\ge0\).  For \(k\ge1\), let\begin{equation}\label{eq:ZkDef}
Z_{k}^{(b'(k))}:=\frac{f_{k,N+1}-f_{k,N}}{p_{N}},
\end{equation} where  \(N\in\mathbb{N}\)  is an integer-valued  random variable independent of  \((U_{i}, i\in\mathbb{Z})\)  such that \(\Pr(N=l)=p_{l}\)  for \(l\geq0\). For simplicity, we will often denote \(Z_{k}^{(b'(k))}\) by \(Z_{k}\). Let \(T_{k}\) be the expected time required to simulate \(Z_{k}\).   Lemma~\ref{le:generalPl} provides  bounds on \(T_{k}\) and on the second moment of   \(Z_{k}\)   and shows that, under certain conditions, \(Z_{k}\) is an unbiased estimator for the negated bias \(\mu-E(f_{k,0})\).
 
\begin{lemma}\label{le:generalPl} For \(k\ge1\), we have \(T_{k}\le3k\sum^{\infty}_{l=0}2^{l}p_{l}\), and
\begin{equation}
\label{eq:ZNSecondMomentGeneralCase}
kE(Z_{k}^{2})\le2(\overline{\nu}(\lfloor{b'(k)}/{2}\rfloor))^{2}(\frac{1}{p_{0}}+\frac{1}{p_{1}})+\sum^{\infty}_{l=2}\frac{2^{3-l}(\overline{\nu}(k2^{l-2})-\overline{\nu}(k2^{l-1}))^{2}}{p_{l}}.
\end{equation}
If the right-hand side of \eqref{eq:ZNSecondMomentGeneralCase} is finite, then \(Z_{k}\) is square-integrable and  \(E(f_{k,0}+Z_{k})=\mu\).      
\end{lemma}
We now detail Step 3. Given \(q\in(0,1]\), let \(Z'_{k}\) be a copy of \(Z_{k}\) independent of \(f_{k,0}\) and let \(Q\)  be a binary variable independent of \((f_{k,0},Z'_{k})\) such that \(\Pr(Q=1)=q\).  Set
\begin{equation}\label{eq:defHatfk}
\hat f_{k}:=f_{k,0}+q^{-1}QZ'_{k}.
\end{equation} In other words, \(\hat f_{k}\) is constructed by sampling \(f_{k,0}\) once and sampling a copy of \(Z_{k}\) with frequency \(q\).  By Lemma~\ref{le:generalPl}, if the right-hand side of \eqref{eq:ZNSecondMomentGeneralCase} is finite, then  \(E(\hat f_{k})=E(f_{k,0})+E(Z_{k})=\mu\), and \(\hat f_{k}\) is an unbiased estimator of \(\mu\). When \(q=1\), copies of  \(f_{k,0}\) and  of \(Z_{k}\) are sampled with the same frequency.  When \(q<1\),  \(Z_{k}\) is sampled less often than \(f_{k,0}\), which can improve the efficiency of \(\hat f_{k}\),   in the same spirit as the Multilevel Monte Carlo Method (MLMC)~\cite{Giles2008} and the randomized dimension reduction algorithm~\cite{kahaRandomizedDimensionReduction20}.   Selecting the \(p_{l}\)'s and \(q\) is studied in Subsections~\ref{sub:Sdependent} and~\ref{sub:Oblivious}. Let  \(\hat T_{k}\) be the expected time to simulate \(\hat f_{k}\). As the expected time to simulate \(f_{k,0}\) is equal to \(k\), we have \(\hat T_{k}=k+qT_{k}\). Note that the estimator \(Z_{k}\) is interesting by itself as it provides an unbiased estimator for the bias of \(f_{k}\) if we set \(b'(k)=b(k)\). We now state the following assumption:
\begin{description}
\item[Assumption B.] There is a positive real number \(w_{0}\) such that  \(k\var(f_{k})\ge w_{0}\) for sufficiently large \(k\).  \end{description}
When \(b(k)=0\), Assumption B can be shown under certain correlation hypotheses  \cite[p. 99]{asmussenGlynn2007}.
\subsection{\(\nu\)-dependent parameters}
\label{sub:Sdependent}
This subsection gives a construction of   \((p_{l},l\geq0)\)   and of \(q\) in terms of \(\bar\nu\). For \(l\geq2\), set  
\begin{equation}\label{eq:plDefBeta}
p_{l}=\frac{\overline{\nu}(k2^{l-2})-\overline{\nu}(k2^{l-1})}{2^{l}\overline{\nu}(k)},
\end{equation} and
\begin{equation}\label{eq:pldefl01}
p_{1}=(1-\sum^{\infty}_{l=2}p_{l})/3 \text{ and }p_{0}=2p_{1}.
\end{equation}Note that \(p_{l}>0\) for \(l\ge2\) since \((\overline{\nu}(i),i\ge0)  \) 
is a strictly decreasing sequence. Furthermore,\begin{eqnarray}\label{eq:pr2^lnudep}
\sum^{\infty}_{l=2}2^{l}p_{l}&=&\sum^{\infty}_{l=2}\frac{\overline{\nu}(k2^{l-2})-\overline{\nu}(k2^{l-1})}{\overline{\nu}(k)}\nonumber\\
&=&1.\end{eqnarray} Hence \(\sum^{\infty}_{l=2}p_{l}\leq1/4\).  Consequently,   \(p_{1}\ge1/4\),     \(p_{0}\ge1/2\), and \((p_{l},l\geq0)\)  is a probability distribution. The \(p_{l}\)'s have been chosen so that the summands in the bounds on \(T_{k}\) and  \(E(Z_{k}^{2})\) in Lemma~\ref{le:generalPl}  are proportional for \(l\geq2\). Lemma~\ref{le:specialPl} shows that \(Z_{k}\) is an unbiased estimator of  \(\mu-E(f_{k,0})\) and provides  bounds on its  second moment and  expected running time. Note that the bound on \(E(Z_{k}^{2})\) is, up to a multiplicative constant, the square of the bound on \(|E(Z_{k})|\) that follows from~\eqref{eq:thLongRunBias} and  the equality \(E(Z_{k})=\mu-E(f_{k,0})\).    
\begin{lemma}\label{le:specialPl} Suppose that A1 holds and that  \((p_{l},l\geq0)\) are given by \eqref{eq:plDefBeta} and \eqref{eq:pldefl01}. For \(k\ge1\), we have \(E(f_{k,0}+Z_{k})=\mu\), \(T_{k}\le9k\), and \begin{equation}
\label{eq:ZNSecondMoment}
kE(Z_{k}^{2})\le20(\overline{\nu}(\lfloor{b'(k)}/{2}\rfloor))^{2}.
\end{equation}    
\end{lemma}
Set
\begin{equation}\label{eq:defq}
q=\frac{\overline{\nu}(\lfloor{b'(k)}/{2}\rfloor)}{\overline{\nu}(0)}.
\end{equation}
Section~\ref{se:motivationforq} gives a motivation for \eqref{eq:defq}. 
\begin{theorem}\label{th:main}
Suppose that A1 holds, that \(k\ge1\), and that  \((p_{l},l\geq0)\) and \(q\) are given by \eqref{eq:plDefBeta}, \eqref{eq:pldefl01} and \eqref{eq:defq}. Then \(\hat f_{k}\) is square-integrable and \(E(\hat f_{k})=\mu\). Moreover,  \(\hat T_{k}\le k+9qk\), and \begin{equation}\label{eq:timeVarianceUBound}
\hat T_{k}\var(\hat f_{k})\leq k\var(f_{k})+8610(\overline{\nu}(0))^{2}\max\left(q,\sqrt{\frac{b'(k)-b(k)}{k}}\right).
\end{equation}  
\end{theorem}
\eqref{eq:timeVarianceUBound} gives a bound on the work-normalized variance of \(\hat f_{k}\) in terms of the work-normalized variance of \(f_{k}\). The constant \(8610\) is an artifact of our calculations. By setting \(b'(k)=\max(b(k),\lceil\sqrt{k}\rceil/2)\), it is easy to check that the second term in the  RHS of  \eqref{eq:timeVarianceUBound} goes to \(0\) as \(k\) goes to infinity. Consequently, under Assumption B, for any given \(\epsilon>0\), we have\begin{equation*}
\hat T_{k}\var(\hat f_{k})\leq (1+\epsilon)k\var(f_{k})
\end{equation*} for sufficiently large \(k\).
In other words, the work-normalized variance of \(\hat f_{k}\) is at most equal to that of \(f_{k}\), up to the multiplicative factor \(1+\epsilon\). Thus, \(\hat f_{k}\) is asymptotically at least as efficient as  \(f_{k}\), as \(k\) goes to infinity,  up to a multiplicative constant arbitrarily close to \(1\).   \comment{\begin{corollary}\begin{displaymath}
b'(k)=\min(k/2,\max(2,b(k),\frac{1}{\xi }\ln(k))
\end{displaymath}   
\end{corollary}
}
\subsection{Oblivious parameters}\label{sub:Oblivious}
When the sequence \(\overline\nu\) is known or can be estimated, the choices of    \((p_{l},l\geq0)\),  of \(q\) and of \(b'(k)\) in Subsection~\ref{sub:Sdependent} yield an  \(\hat f_{k}\) that is asymptotically at least as efficient as  \(f_{k}\). Under certain assumptions and without  explicit knowledge of \(\bar\nu\),   this subsection provides choices of    \((p_{l},l\geq0)\),  of \(q\) and of \(b'(k)\)   so that the work-normalized variance of    \(\hat f_{k}\) is at most equal to that of   \(f_{k}\), up to a multiplicative factor arbitrarily close to \(1\).    We first state the following assumption.  \begin{description}
\item[Assumption 3 (A3).] For  \(l\geq0\),  \begin{equation}\label{eq:obliviouspl}
p_{l}=\frac{1}{\theta(l)2^{l}}-\frac{1}{\theta(l+1)2^{l+1}},
\end{equation} where \(\theta\) is an increasing function on \([0,\infty)\), with 
 \(\theta(x)=1\) for \(x\in[0,1]\),  
\begin{equation}\label{eq:infiniteSumAssumptionZeta}
\sum^{\infty}_{l=0}\frac{1}{\theta^{}(l)}<\infty.
\end{equation} Furthermore, Assumption A1 holds and \begin{equation}\label{eq:infiniteSumAssumptionZetaVu}
\sum^{\infty}_{i=0}\sqrt{\frac{\nu(i)\theta(\log_{2}(4i+1))}{i+1}}<\infty.
\end{equation} \end{description}
Observe that the \(p_{l}\)'s given in \eqref{eq:obliviouspl} depend only on \(\theta\), and that  \eqref{eq:infiniteSumAssumptionZetaVu} is a stronger version of 
 \eqref{eq:infiniteSumAssumption}. Standard calculations show the following.   

\begin{example}\label{ex:exponentialDist}Suppose that, for some positive constants \(c\), \(\xi\) and \(\delta\) with \(\delta<\xi-1\), Assumption A1 holds with \(\nu(i)= c(i+1)^{-\xi}\) for \(i\geq0\),  and that the \(p_{l}\)'s are given by \eqref{eq:obliviouspl},  with  
 \(\theta(x)=1\) for \(x\in[0,1]\), and \(\theta(x)=2^{\delta(x-1)}\) for \(x\geq1\). Then Assumption A3 holds and \(p_{l}\) is of order \(2^{-(\delta+1)l}\).   \end{example}
Distributions with exponentially decreasing tails have  been previously used in  RMLMC pricing of financial derivatives \cite{GlynnRhee2015unbiased,kahale2020Asian}. In Example~\ref{ex:exponentialDist}, the choice of the \(p_{l}\)'s depends  on \(\xi\) because of the condition \(\delta<\xi-1\). Example \ref{ex:thetaPowerExample} shows that  the \(p_{l}\)'s can chosen without any knowledge on \(\xi\).

\begin{example}\label{ex:thetaPowerExample} Suppose that, for some positive constants \(c\) and \(\xi\) with \( \xi>1\), Assumption A1 holds with \(\nu(i)= c(i+1)^{-\xi}\) for \(i\geq0\),    and that the \(p_{l}\)'s are given by \eqref{eq:obliviouspl},  with \(\theta(x)=\max(1,x)^{\delta}\) for \(x\geq0\), where \(\delta>1\).   Then Assumption A3 holds and \(p_{l}\) is of order \(l^{-\delta}2^{- l}\).\end{example}
When Assumption A2 holds, Example \ref{ex:exponentialDistA2} shows that the \(p_{l}\)'s can be chosen as in Example~\ref{ex:exponentialDist} without any further knowledge on \(\nu\).

\begin{example} \label{ex:exponentialDistA2}Suppose that Assumption A2 holds  and that the \(p_{l}\)'s are given by \eqref{eq:obliviouspl},  with  
 \(\theta(x)=1\) for \(x\in[0,1]\), and \(\theta(x)=2^{\delta(x-1)}\) for \(x\geq1\), where \(\delta\)   is a positive constant. Then Assumption A3 holds.    \end{example}
Suppose now that Assumption A3 holds.   
For \(j\ge0\), let
\begin{equation*}
\overline{\nu}_{\theta}(j):=\sum^{\infty}_{i=j}\sqrt{\frac{\nu(i)\theta(\log_{2}(4i+1))}{i+1}}.
\end{equation*}
Assumption A3 shows that \(\overline{\nu}_{\theta}(j)\) is finite and goes to \(0\) as \(j\) goes to infinity, and that   \(\overline{\nu}(j)\le \overline{\nu}_{\theta}(j)\) for \(j\geq0\). 
Lemma~\ref{le:ObliviousPl} shows that, under Assumption A3, \(Z_{k}\) is an unbiased estimator of  \(\mu-E(f_{k,0})\), and provides  a bound on its  second moment and on \(T_{k}\).
\begin{lemma}\label{le:ObliviousPl} Suppose that Assumption A3  holds.  Then, for  \(k\geq1\), we have \(E(f_{k,0}+Z_{k})=\mu\) and \begin{equation}
\label{eq:ZNSecondMomentOblivious}
kE(Z_{k}^{2})\le28(\overline{\nu}_{\theta}(\lfloor{b'(k)}/{2}\rfloor))^{2}.
\end{equation}Furthermore, \(T_{k}\le3k\sum^{\infty}_{l=0}1/\theta(l)\).     
\end{lemma}
 Theorem~\ref{th:maintheta} shows that, under Assumption A3,  \(\hat f_{k}\) is an unbiased estimator of \(\mu\) and gives bounds on its running time and variance.   
\begin{theorem}\label{th:maintheta}
Suppose that Assumption A3  holds. Then, for \(k\ge1\), \(\hat f_{k}\) is square-integrable, \(E(\hat f_{k})=\mu\), and \begin{equation}\label{eq:varfHatkOblivious}
k\var(\hat f_{k})\leq k\var(f_{k})+796(\overline{\nu}(0))^{2}\sqrt{\frac{b'(k)-b(k)}{k}}+\frac{28}{q}(\overline{\nu}_{\theta}(\lfloor{b'(k)}/{2}\rfloor))^{2}.
\end{equation} Moreover, \(\hat T_{k}\le k+3(\sum^{\infty}_{l=0}1/\theta(l))qk\). \end{theorem}
Observe that the second (resp. last) term in the right-hand side of  \eqref{eq:varfHatkOblivious} is an increasing (resp. decreasing) function of \(b'(k)\). Likewise,   the bound on the variance (resp. running time) of  \(\hat f_{k}\)  is a decreasing (resp. increasing) function of \(q\). 
Theorem~\ref{th:maintheta} shows that setting
\begin{equation*}q=\frac{\epsilon}{3\sum^{\infty}_{l=0}1/\theta(l)}, 
\end{equation*} where \(\epsilon\in(0,1)\),  ensures that \(\hat T_{k}\le k(1+\epsilon)\). Furthermore,  if \(b'(k)=\max(b(k),\lceil\sqrt{k}\rceil/2)\), then \(k\var(\hat f_{k})\leq k\var(f_{k})+\epsilon\) for sufficiently large \(k\). This is because \(\overline{\nu}_{\theta}(j)\) goes to \(0\) as \(j\) goes to infinity.
 Then,  under Assumption B, for any given \(\epsilon'>0\), if \(\epsilon\) is sufficiently small and \(k\) sufficiently large, we have\begin{equation*}
\hat T_{k}\var(\hat f_{k})\leq (1+\epsilon')k\var(f_{k}).
\end{equation*} Here again, \(\hat f_{k}\) is asymptotically at least as efficient as  \(f_{k}\), as \(k\) goes to infinity,  up to a multiplicative factor arbitrarily close to \(1\). In practice, in the absence of precise knowledge on the behavior of the chain, setting \(b'(k)=\max(b(k),\lfloor \epsilon''k\rfloor)\), where \(\epsilon''\in(0,1/2]\), e.g., \(\epsilon''=0.1\), would make \(b'(k)\) reasonably large without deleting too many observations.

Under Assumption A2, and for specific values of the \(p_{l}\)'s,   Theorem~\ref{th:ObliviousPolynomialZeta} gives a bound on the  variance of \(\hat f_{k}\) that depends explicitly on \(c\) and \(\xi\). It also provides an improved variance bound  under an additional decay  assumption on \(\nu\).         
\begin{theorem}\label{th:ObliviousPolynomialZeta}
Suppose that Assumption A2  holds, and that the \(p_{l}\)'s are given by \eqref{eq:obliviouspl},  with \(\theta(x)=\max(1,x)^{\delta}\) for \(x\geq0\), where \(\delta\in(1,2]\). Then, for \(k\ge1\),   Assumption A3  holds, \begin{equation}\label{eq:A2T^kbound}
\hat T_{k}\le k(1+\frac{9q}{\delta-1}),
\end{equation}  and 
 \begin{equation}\label{eq:kVarHatfkBoundExpoNu}
k\var(\hat f_{k})\leq k\var(f_{k})+\frac{Ac}{\xi}\sqrt{\frac{b'(k)-b(k)}{k}}+\frac{Ac}{q\xi}\min\left(\ln^{\delta}\left(\frac{3}{\xi}\right),\frac{e^{-\xi b'(k)/2}}{\xi}\right),
\end{equation} where \(A\) is an absolute constant. Moreover, if \(\nu(i)\le c/(i+1)\) for \(i\geq0\), then\begin{equation}\label{eq:kVarHatfkBoundExpoGoemNu}
k\var(\hat f_{k})\leq k\var(f_{k})+A'c\ln^{2}\left(\frac{3}{\xi}\right)\sqrt{\frac{b'(k)-b(k)}{k}}+\frac{A'c}{q}\min\left(\ln^{\delta+2}\left(\frac{3}{\xi}\right),\frac{e^{-\xi b'(k)/2}}{\xi^{2}}\right),
\end{equation}where \(A'\) is an absolute constant.\end{theorem}
The second term in the RHS of \eqref{eq:kVarHatfkBoundExpoNu} is of order \(1/\xi\), while the second term in the RHS of \eqref{eq:kVarHatfkBoundExpoGoemNu} has a logarithmic dependence on \(\xi\). Both terms can be made arbitrarily small by setting
\(b'(k)=\max(b(k),\lfloor \epsilon k\rfloor)\), with  \(\epsilon\in(0,1/2]\) sufficiently small.  For fixed \(q\), the last term in the RHS of \eqref{eq:kVarHatfkBoundExpoNu}  is  uniformly bounded by a term of order  \(1/\xi\), up to a polylogarithmic factor, while  the last term in the RHS of \eqref{eq:kVarHatfkBoundExpoGoemNu} is  uniformly bounded by a term with a logarithmic  dependence on \(\xi\).
When \(b'(k)\) is proportional to \(k\), both terms decrease exponentially with \(k\). 
\begin{remark} The results of Theorem~\ref{th:ObliviousPolynomialZeta} are still valid if the constraint  \(\delta\in(1,2]\) is replaced with  \(\delta\in(1,\delta_{0}]\), for any fixed \(\delta_{0}>1\), and  if \(A\) and \(A'\)  and the constant \(9\) in  \eqref{eq:A2T^kbound} are replaced with constants that depend on \(\delta_{0}\). 
\end{remark}
\subsection{A stratified unbiased  estimator}\label{sub:strat}
Given \(n,k\ge1\) and \(q\in(0,1)\), let
\begin{equation*}
\tilde f_{k,n}:=\tilde f_{k}+\tilde Z_{k},
\end{equation*} 
where \(\tilde f_{k}\) is the average of \(n\) independent copies of \(f_{k,0}\) and \(\tilde Z_{k}\) is the average of \(\lceil nq\rceil\) independent copies of \(Z_{k}\). The estimator  \(\tilde f_{k,n}\) is a stratified version of \(\hat f_{k}\) and has similar properties. By Lemma~\ref{le:ObliviousPl}, under Assumption 3, \begin{displaymath}
E(\tilde f_{k,n})=E(f_{k,0})+E(Z_{k})=\mu.
\end{displaymath}
Furthermore,
\begin{equation*}
\var(\tilde f_{k,n})=\frac{\var(f_{k,0})}{n}+\frac{\var(Z_{k})}{\lceil nq\rceil}.
\end{equation*}
On the other hand, it follows from the definition of   \(\hat f_{k}\) that 
\begin{eqnarray*}
\var(\hat f_{k})
&=&\var(f_{k,0})+q^{-2}\var(QZ'_{k})\\
&=&\var(f_{k,0})+q^{-1}E(Z_{k}^{2})-E(Z_{k})^{2}\\
&\ge&\var(f_{k,0})+q^{-1}\var(Z_{k}).
\end{eqnarray*}Thus, \(n\var(\tilde f_{k,n})\leq\var(\hat f_{k})\). The expected time to simulate \(\tilde f_{k,n}\) is \(\tilde T_{k,n}=nk+\lceil nq\rceil T_{k}\). Thus  \(\tilde T_{k,n}\le n\hat T_{k}+ T_{k}\) and  \(\tilde T_{k,n}\var(\tilde f_{k,n})\leq( \hat T_{k}+ T_{k}/n)\var(\hat f_{k})\).  Consequently, as \(n\) goes to infinity, the estimator \(\tilde f_{k,n}\) is asymptotically at least as efficient as  \(\hat f_{k}\).   
\subsection{Implementation details}\label{sub:implementationDetails}
\begin{algorithm}[H]
\caption{Procedure LongRun}
\label{alg:LR}
\begin{algorithmic}[1]
\Procedure{LR}{$B,K,X[0],\dots,X[h],S[0],\dots,S[h]$}
\For{\(j\gets 0,h\)} 
\State \(
S[j]\gets 0
\)
\EndFor
\For{\(i\gets 0,K-1\)} 
\State Simulate \(V\sim U_{0}\)
\For{\(j\gets 0,h\)} 
\If{\(i\geq B\)}
\State \(S[j]\gets S[j]+ f(X[j])\)
\EndIf
\State \(X[j]\gets g(X[j],U)\)
\EndFor
\EndFor
\EndProcedure
\end{algorithmic}
\end{algorithm}

\begin{algorithm}[H]
\caption{Procedure Bias}
\label{alg:Bias}
\begin{algorithmic}[1]
\Procedure{BIAS}{$k,b'(k),(p_{l},l\geq0)$}
\State Simulate a random variable \(N\) such that \(\Pr(N=l)=p_{l}\) for \(l\in\mathbb{N}\)
\State \(X[0]\gets X_{0}\)
\State{LR}($0,k2^{N},X[0],S[0]$)
\State \(X[1]\gets X_{0}\)
\State{LR}(\(k(2^{N}-1)+b'(k),k2^{N},X[0],X[1],S[0],S[1]\))
\State \Return{\(\frac{1}{p_{N}(k-b'(k))}(S[1]-S[0])\)} 
\EndProcedure
\end{algorithmic}
\end{algorithm}

\begin{algorithm}[H]
\caption{Procedure UnbiasedLongRun}
\label{alg:RHAlk}
\begin{algorithmic}[1]
\Procedure{ULR}{$k,b'(k),(p_{l},l\geq0),q$}
\State \(X[0]\gets X_{0}\)
\State{LR}($b'(k),k,X[0],S[0]$)
\State $f_{k,0}\gets\frac{1}{k-b'(k)}S[0]$
\State Sample \(W\) uniformly from \([0,1]\)
\If{$W>q$} 
\State \Return $f_{k,0}$
\Else
\State \Return{\(f_{k,0}-\frac{1}{q}{\text{Bias}} (k,b'(k),(p_{l},l\geq0))\)} \EndIf
\EndProcedure
\end{algorithmic}
\end{algorithm}

\begin{algorithm}[H]
\caption{Procedure StratifiedUnbiasedLongRun}
\label{alg:SULR}
\begin{algorithmic}[1]
\Procedure{SULR}{$n,k,b'(k),(p_{l},l\geq0),q$}
\State $S'\gets0$
\For{\(i\gets1,n\)}
\State \(X[0]\gets X_{0}\)
\State{LR}($b'(k),k,X[0],S[0]$)
\State $S'\gets S'+\frac{1}{k-b'(k)}S[0]$
\EndFor
\State $S''\gets0$
\For{\(i\gets1,\lceil nq\rceil\)}
\State $S''\gets S''+{\text{Bias}}(k,b'(k),(p_{l},l\geq0))$
\EndFor
\State \Return{\(S'/n-S''/\lceil nq\rceil\)} 
\EndProcedure
\end{algorithmic}
\end{algorithm}
Algorithm~\ref{alg:LR} assumes that \(B\) and \(K\) are integers with \(0\leq B<K\). The arguments \(X[0],\dots,X[h]\) and \(S[0],\dots,S[h]\) of LR are real numbers passed by reference, i.e., modifications made to these arguments
in LR have effect in any procedure that calls LR. For \(0\leq j\leq h\), denote by \(X_{0}[j],\dots,X_{K}[j]\) the successive values of \(X[j]\) during the execution of Algorithm~\ref{alg:LR}, where   \(X_{0}[j]\) is the value of  \(X[j]\) at the beginning of LR. It is assumed that the \(V\)'s generated in LR and  \((X_{0}[0],\dots,X_{0}[h])\) are independent random variables.

Under Assumption A3, Lemma~\ref{le:ObliviousPl} shows that   \(E(Z_{k}^{(b'(k))})=\mu-E(f_{k,0})\) for  \(k\geq1\). Thus \(-Z_{k}^{(b'(k))}\) is an unbiased estimator of the bias of \(f_{k,0}\).  Algorithm~\ref{alg:Bias}   
provides a detailed implementation of \(-Z_{k}^{(b'(k))}\) based on \eqref{eq:fklDef} and \eqref{eq:ZkDef} and on  the procedure LR. Algorithm~\ref{alg:RHAlk} gives a detailed implementation for  \(\hat f_{k}\) based on LR and on the procedure BIAS in  Algorithm~\ref{alg:Bias}.   Proposition~\ref{pr:implementationDetails} shows that  the outputs of Algorithms~\ref{alg:Bias}   
 and~\ref{alg:RHAlk} are consistent with the definitions of \(Z_{k}^{(b'(k))}\) and of  \(\hat f_{k}\).     
\begin{proposition}\label{pr:implementationDetails}The random variable output by the procedure BIAS (resp. ULR) has the same distribution as  \(-Z_{k}^{(b'(k))}\) (resp. \(\hat f_{k}\)).  
\end{proposition}
Note that the procedure BIAS can be used to estimate the bias of \(f_{k}\) by setting \(b'(k)=b(k)\). The procedure SULR in Algorithm~\ref{alg:SULR} provides an implementation of \(\tilde f_{k,n}\).

\section{Examples}\label{se:examples}
\subsection{GARCH volatility model}\label{sub:GARCH}
In the  GARCH(1,1) volatility model (see~\citep[Ch. 23]{Hull14}), the daily volatility  \(\sigma_{i}\)  of an index or exchange rate, calculated at the end of day  \(i\), satisfies the following recursion:
\begin{equation*}
{\sigma_{i+1}}^{2}= w +\alpha{\sigma_{i}}^{2}U_{i}^{2}+\beta {\sigma_{i}}^{2},
\end{equation*}
 \(i\geq0\), where \( w \), \(\alpha\) and \(\beta\) are positive constants with \(\alpha+\beta<1\), and \((U_{i}, i\geq0)\) are independent standard Gaussian random variables.  At the end of day \(0\), given  \(\sigma_{0}\geq0\) and a real number \(z\), we want to estimate \(\lim_{i\rightarrow\infty}\Pr(\sigma_{i}^{2}> z)\), if such a limit exists. In this example, \(F=F'=\mathbb{R}\), with \(X_{i}=\sigma_{i}^{2}\) and \(g(x,u)=w +\alpha xu^{2}+\beta x\), and   \(f(u)={\bf1}\{u> z\}\) for \(u\in\mathbb{R}\). The proof of \cite[Proposition 9]{kahaRandomizedDimensionReduction20} implies \eqref{eq:rhoDefx} with \(x=0\) and 
\(\nu'(i)= c(\alpha+\beta)^{i/2}\) for \(i\geq0\), for some constant \(c\). Thus,
 Assumption A2 holds. 
 \subsection{\(GI/G/1\) queue}\label{sub:GIG1example}
 Consider  a  \(GI/G/1\) queue
 where   customers are served by a single server in order of arrival.  For \(n\geq0\), let \(A_{n}\), \(V_{n}\) and \(X_{n}\) be the arrival time, service time and waiting time (exclusive of service time) of customer \(n\).    For \(n\geq0\), define the interarrival time  \(D_{n}:=A_{n+1}-A_{n}\). Assume that the system starts empty at time \(0\), that   \((D_{n},V_{n})\), \(n\geq0\), are identically distributed, and that the random variables \(\{D_{n},V_{n},n\geq0\}\) are independent.
The waiting times satisfy the Lindley recursion \cite[p.1]{asmussenGlynn2007} \begin{equation*}
X_{i+1}=\max(0,X_i+U_{i}),
\end{equation*}where \(U_{i}:=V_i-D_{i}\) for \(i\geq0\), with \(X_{0}=0\). We want to estimate \(\lim_{i\rightarrow\infty}E(X_{i})\), if such a limit exists. Here, we have \(F=F'=\mathbb{R}\), with \(g(x,u)=\max(0,x+u)\), and   \(f\) is the identity function. 
Proposition~\ref{pr:GG1} below shows that Assumption A2 holds under suitable  conditions. 
\begin{proposition}\label{pr:GG1}If  there are constants \(\gamma>0\) and \(\eta<1\) such that
\begin{equation}\label{eq:momentCondGD1}
E(e^{\gamma U_{i}})\leq\eta
\end{equation}  for \(i\ge0\), then Assumption A1 holds when   \(f\) is the identity function, with  \(\nu(i)= \gamma'\eta^{i}\)  for \(i\ge0\), where \(\gamma'\) is a constant. \end{proposition}
 The proof of Proposition~\ref{pr:GG1} is very similar to that of \cite[Proposition 10]{kahaRandomizedDimensionReduction20}, and is omitted. The condition  \eqref{eq:momentCondGD1} is related to the stability condition \(E(U_{i}<0)\), and is justified in \cite{kahaRandomizedDimensionReduction20}.  

Our approach can also  estimate \(\lim_{i\rightarrow\infty}\Pr(X_{i}> z)\), where  \(z\) is a real number, under suitable conditions. In this case, \(F\), \(F'\) and \(g\) are the same as above, and   \(f(u)={\bf1}\{u> z\}\) for \(u\in\mathbb{R}\).
 
\begin{proposition}\label{pr:GG1Binary}If \(E(U_{i}<0)\) and \(E(U_{i}^{6})\) is finite, then Assumption A1 holds when   \(f(u)={\bf1}\{u> z\},\)  with  \(\nu(i)=c(i+1)^{-2}\)  for \(i\ge0\), where \(c\) is a constant. \end{proposition}
\subsection{High-dimensional Gaussian vectors}
Let \(V\)  be a \(d\times d\) positive definite matrix with all diagonal entries equal to \(1\). The standard algorithm to generate a Gaussian vector with covariance matrix \(V\) is based on  the Cholesky decomposition, that takes \(O(d^{3})\)  time.  \citet{kahaleGaussian2019} describes an alternative method that approximately simulates a centered
\(d\)-dimensional Gaussian vector \(X\) with covariance matrix \(V\). Let \(j\) be a random integer uniformly distributed 
 in \(\{1,\ldots ,d\}\),  and let \(e\) be  the  \(d\)-dimensional  random column vector whose
\(j\)-th coordinate is \(1\) and remaining coordinates are \(0 \). Let \((e_{i},i\geq0)\) be a sequence of independent copies of \(e\), and let   \((g_{i},i\geq0)\)
 be a sequence of independent standard Gaussian random variables, independent of \((e_{i},i\geq0)\).   Define the Markov chain of \(d\)-dimensional column vectors  \((X_{i},i\ge0)\)
as follows. Let  \(X_{0}=0\) and, for   \(i\geq0\),  let
 \begin{equation*}
X_{i+1}=X_{i}+(g_{i}- e_{{i}}^{T}X_{i})(V  e_{{i}}).
\end{equation*} In this example, \(F\) is the set of   \(d\)-dimensional  column vectors, \(F'=\mathbb{R}\times F\), with \(U_{i}=(g_{i},e_{i})\) and \(g(x,g',e')=x+(g'-e'^{T}x)(Ve')\) for \((g',e')\in F'\).

 \begin{theorem}\label{th:Gaussian}Let \(\hat \kappa\) and \(\hat \gamma\) be two positive constants with  \(\hat \gamma\le1\). Consider a real-valued Borel function \(f\) of \(d\) variables such that\begin{equation}\label{eq:defkappaLipsG}
E((f(X)-f(X'))^{2})\leq \hat\kappa^{2}( E(||X-X'||^{2}))^{\hat\gamma}
\end{equation}
for any centered Gaussian column vector \(\begin{pmatrix}X \\
X' \\
\end{pmatrix}\)   with     \(\cov(X)\le V\) and
  \(\cov(X')\le V\), where \(X\) and \(X'\) have dimension \(d\).
Then Assumption  A1 holds for the function \(f\), with  \begin{equation*}
\nu(i)= \hat\kappa^{2}\min((\lambda_{\max}d)^{\hat\gamma}(1-\frac{\lambda_{\min}}{d})^{\hat\gamma i}, (\frac{d^{2}}{i+1})^{\hat\gamma}),
\end{equation*} \(i\geq0\), where \(\lambda_{\max}\) (resp. \(\lambda_{\min}\)) is the largest (resp. smallest) eigenvalue of \(V\).  \end{theorem} 
Since \(\tr(V)=d\), we have \(\lambda_{\max}\leq d\). As \(1+x\leq e^{x}\) for \(x\in\mathbb{R}\),  Theorem~\ref{th:Gaussian} shows that Assumption A2 holds with \(c=\hat\kappa^{2}d^{2\hat\gamma}\) and \(\xi=\lambda_{\min}\hat\gamma/d\). By Theorem~\ref{th:biasedLongRun}, \(E(f(X_{h}))\) has a finite limit \(\mu\)  as \(h\) goes to infinity.  It follows from  \cite[Theorem~4]{kahaleGaussian2019} that \(\mu=E(f(X))\), where \(X\) is a \(d\)-dimensional column vector with \(X\sim N(0,V)\). 

Assume now that \(\hat\gamma=1\) and that  \(b(k)=0\).  Theorem~\ref{th:Gaussian} shows that   \(\nu(i)\le c/(i+1)\) for \(i\geq0\).  Combining \eqref{eq:thMSE} and~\eqref{eq:deltaBoundGeom} yields  \(kE((f_{k}-\mu )^{2})\le5096c\ln^{2}(2/\xi)\), which is weaker   than the bound \begin{equation*}
kE((f_{k}-\mu )^{2})\le18c
\end{equation*}
of \cite[Theorem~2]{kahaleGaussian2019} by a polylogarithmic factor. Note that Theorem~\ref{th:ObliviousPolynomialZeta} is applicable in this example.
\comment{
Suppose now that the \(p_{l}\)'s are given by \eqref{eq:obliviouspl},  with \(\theta(x)=\max(1,x)^{\delta}\) for \(x\geq0\), where \(\delta\in(1,2]\), and that \(q=(\delta-1)\epsilon/9\) and  \(b'(k)=\lfloor \epsilon'k\rfloor\), where \(\epsilon\in(0,1)\) and \(\epsilon'\in(0,1/2]\), with \(k\ge1\).
Theorem~\ref{th:ObliviousPolynomialZeta} implies that  \(\hat T_{k}\le k(1+\epsilon)\) and that~\eqref{eq:kVarHatfkBoundExpoGoemNu} holds.
}

\comment{\item This question does not use the same notation as in the paper. Suppose that $h$ is a $\mu$-strongly convex on \(\mathbb{R}^{d}\) and  $h=E(g)$, where $g$ is a random convex  \(L\)-smooth function. If \(x\) and \(y\) are in \(\mathbb{R}^{d}\), and \begin{displaymath}
x'=x-\alpha g'(x),
\end{displaymath} and
 \begin{displaymath}
y'=y-\alpha g'(y),
\end{displaymath}
where \(\alpha L \le 1/4\), say, can we prove that \begin{displaymath}
E(||x'-y'||^2)\le(1-\alpha\beta)||x-y||^2,
\end{displaymath}
for some \(\beta>0\) independent of \(\alpha\)? This property holds in the linear regression example. If it holds in general, this could be used to generalize the linear regression example.
}
\section{Numerical experiments}\label{se:NumerExper}
The codes in our simulation experiments were written in the C++ programming language. We assume that the \(p_{l}\)'s are determined as in Example \ref{ex:exponentialDist}, with \(\delta=1/2\). We set \(q=(3\sum^{\infty}_{l=0}2^{l}p_{l})^{-1}\).
By Lemma \ref{le:generalPl} and the discussion thereafter, for \(k\ge1\), we have \(\hat T_{k}\leq 2k\). Thus, on average, at most half of the running time of \(\hat f_{k}\) is devoted to the simulation of \(Z_{k}\).  Fig.~\ref{fig:GARCHMH1GG1} shows the standard deviation per replication and the absolute value of the bias of \(f_{k}\), i.e.,  \(|E(f_{k})-\mu|\),   estimated from \(10^{6}\) independent replications of \(f_{k}\) and of \(Z_{k}\), respectively, with   \(b(k)=b'(k)=\lfloor k/10\rfloor\). The absolute value of the bias is also reported in Tables \ref{tab:BiasGARCH0.1}, \ref{tab:BiasMH0.1} and \ref{tab:BiasGG0.1}. In these tables, ``Std'' and ``Cost'' refer to the standard deviation and running time of a single iteration of \(Z_{k}\), respectively.   Tables~\ref{tab:GARCH0.1}, \ref{tab:GARCH0.5}, \ref{tab:MH0.1}, \ref{tab:MH0.5}, \ref{tab:GG0.1} and~\ref{tab:GG0.5} compare the methods LR, ULR, and SULR with burn-in periods  \(b(k)=b'(k)=\lfloor k/10\rfloor\) and   \(b(k)=b'(k)=\lfloor k/2\rfloor\).  The methods LR and ULR were implemented with \(10^{6}\) independent replications, and the  method SULR was implemented with  parameter \(n=10^{6}\). A \(95\%\) confidence interval was calculated for \(\mu\) by the methods ULR and SULR. For the methods LR and ULR,   ``Std'' refers to  the standard deviation per replication, whereas   ``Std'' refers to the standard deviation of the output for the method SULR. For the method LR, the root mean square error ``RMSE'' is calculated using the formula \(\text{RMSE}=\sqrt{\text{Std}^{2}+\text{Bias}^{2}}\), where the  bias is estimated from \(10^{5}\) independent replications of \(Z_{k}\).   For the methods ULR and SULR,  RMSE = Std. For the methods LR and ULR, the variable ``Cost'' is the average number of times the Markov chain is simulated per replication, that is, the average number of calls to the function \(g\). For the methods SULR, the variable ``Cost'' is the total number  of calls to the function \(g\). Finally, the mean square error is \(\text{MSE}=\text{RMSE}^2\). In other words, the mean square error is equal to the sum of bias squared and variance. For a fixed computing budget, the mean square error is a standard measure of the performance of a biased estimator \cite{glasserman2004Monte}. Following \cite{GlynnRhee2015unbiased}, we measure the performance of a method through the product \(\text{Cost}\times\text{MSE}\): this product is low when the performance is high.
\begin{figure}
\centering
\begin{subfigure}[b]{0.3\textwidth}
\centering
\begin{tikzpicture}[scale=.5]
\begin{loglogaxis}[
    title={GARCH},
    xlabel={$k$},
    ylabel={},
    legend pos=south west,
    ymajorgrids=true,
    grid style=dashed,
]
 \addplot[
    color=magenta,
    mark=o,
    ]
    coordinates {
(25,0.368671)
(50,0.287042)
(100,0.164856)
(200,0.0676364)
(400,0.0188021)
(800,0.00272414)
(1600,0.000115529)
(3200,4.38194e-007)
  };
 \addplot[
    color=red,
    mark=square,
    ]
    coordinates {
(25,0.0986092)
(50,0.180494)
(100,0.222882)
(200,0.207181)
(400,0.16306)
(800,0.119451)
(1600,0.0855771)
(3200,0.0607933)
};
\legend{\(\text{Bias}\), \(\text{Std}\)}
\end{loglogaxis}
\end{tikzpicture}
\label{fig:sonarlambda1}
\end{subfigure}
\begin{subfigure}[b]{0.3\textwidth}
\centering
\begin{tikzpicture}[scale=.5]
\begin{loglogaxis}[
    title={\(M/H_{k}/1\) queue},
    xlabel={$k$},
    ylabel={},
    legend pos=south west,
    ymajorgrids=true,
    grid style=dashed,
]
 \addplot[
    color=magenta,
    mark=o,
    ]
    coordinates {
(25,4.41452)
(50,3.15822)
(100,1.95343)
(200,0.958234)
(400,0.364616)
(800,0.102446)
(1600,0.017811)
(3200,0.00131372)
  };
 \addplot[
    color=red,
    mark=square,
    ]
    coordinates {
(25,4.01376)
(50,4.93301)
(100,5.59125)
(200,5.75474)
(400,5.23332)
(800,4.21438)
(1600,3.13942)
(3200,2.2548)
};
\legend{\(\text{Bias}\), \(\text{Std}\)}
\end{loglogaxis}
\end{tikzpicture}
%\hspace{1cm}
%\caption{Convergence on Madelon dataset with \(\lambda=\bar L/n\)}
\label{fig:madelonLambda1}
\end{subfigure}
\begin{subfigure}[b]{0.3\textwidth}
\centering
\begin{tikzpicture}[scale=.5]
\begin{loglogaxis}[
    title={\(GI/G/1\) queue},
    xlabel={$k$},
    ylabel={},
    legend pos=south west,
    ymajorgrids=true,
    grid style=dashed,
]
 \addplot[
    color=magenta,
    mark=o,
    ]
    coordinates {
(25,0.206705)
(50,0.131739)
(100,0.0701314)
(200,0.030154)
(400,0.0099214)
(800,0.00231302)
(1600,0.000342867)
(3200,2.13785e-005)
};
 \addplot[
    color=red,
    mark=square,
    ]
    coordinates {
(25,0.210041)
(50,0.239441)
(100,0.232984)
(200,0.203241)
(400,0.162579)
(800,0.121961)
(1600,0.0883594)
(3200,0.0630348)
};
\legend{\(\text{Bias}\), \(\text{Std}\)}
\end{loglogaxis}
\end{tikzpicture}
\end{subfigure}
\caption{Absolute bias and standard deviation of time-average estimators with \(10^{6}\) independent replications and burn-in period  \(b(k)=\lfloor k/10\rfloor\).}
\label{fig:GARCHMH1GG1}
\end{figure}
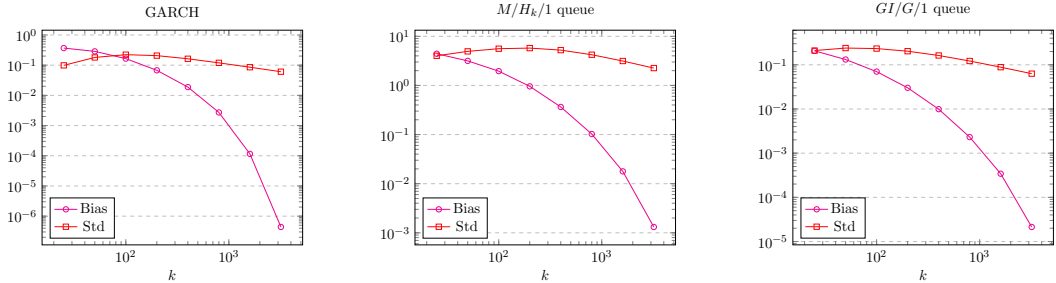
\subsection{GARCH volatility model}
Tables~\ref{tab:GARCH0.1} and \ref{tab:GARCH0.5} estimate  \(\lim_{i\rightarrow\infty}\Pr(\sigma_{i}^{2}> z)\), with   \(\alpha=0.05\), \(\beta=0.92\),   \({\sigma_{0}}^{2}=2\times10^{-5}\), \(w=1.2\times10^{-6}\) and \(z=4\times10^{-5}\).    The left panel in Fig.~\ref{fig:GARCHMH1GG1} shows that, for small values of \(k\), the bias and standard deviation of \(f_{k}\)  are of the same order of magnitude, while for large values of \(k\), the bias is much smaller than the standard deviation and decays at a faster rate. This is consistent with the discussion preceding  Lemma~\ref{le:StdSum}.  In Tables~\ref{tab:GARCH0.1} and \ref{tab:GARCH0.5}, because of the choice of \(q\),  the total running time of   ULR and of SULR is about twice that of LR. For small values of  \(k\), the product \(\text{Cost}\times\text{MSE}\) is much smaller for LR than for ULR and SULR but  LR exhibits a strong bias.  For large values of \(k\), the product \(\text{Cost}\times\text{MSE}\) is twice as large for ULR and SULR as for LR. This is due to the fact that, because of the choice of \(q\), about half the running time of   ULR and SULR is devoted to estimating the bias, that is negligible for large values of  \(k\). The performance of ULR and of  SULR tends to increase with \(k\). This can be explained by the diminishing contribution of the bias to the work-normalized variance of  ULR and of  SULR   as \(k\) increases.  Finally,  ULR and SULR perform better in  Table~\ref{tab:GARCH0.1} than in Table~\ref{tab:GARCH0.5} for large  values of \(k\), and the reverse effect is observed for small values of \(k\). This can be explained by the fact that the variance (resp. bias) of \(f_{k}\) tends to be low (resp. high) when the burn-in is small \cite{whitt1991longRun}, and the impact of the bias on the performance of ULR and of  SULR diminishes as \(k\) increases.  
\begin{table}
\caption{Absolute value of bias in estimating \(\lim_{i\rightarrow\infty}\Pr(\sigma_{i}^{2}> z)\) in a GARCH volatility model with  burn-in period  \(b(k)=\lfloor k/10\rfloor\) and \(10^{6}\) independent replications.}
\begin{center}
\begin{small}
\begin{tabular}{rrccc}\hline
$k$ &  burn-in   & $95\%$ confidence interval  & Std  & Cost \\ \hline
$25$&$2 $ &  $3.6\times10^{-1} \pm 1\times10^{-3}$ & $5.9\times10^{-1}$ & $2.06\times10^2$ \\
$50$&$5 $ &  $2.8\times10^{-1} \pm 8\times10^{-4}$ & $4.1\times10^{-1}$ & $4.14\times10^2$ \\
$100$&$10 $ &  $1.6\times10^{-1} \pm 5\times10^{-4}$ & $2.6\times10^{-1}$ & $7.97\times10^2$ \\
$200$&$20 $ &  $6.7\times10^{-2} \pm 2\times10^{-4}$ & $1.1\times10^{-1}$ & $1.65\times10^3$\\
$400$&$40 $ &  $1.8\times10^{-2} \pm 7\times10^{-5}$ & $3.3\times10^{-2}$ & $3.21\times10^3$\\
$800$&$80 $ &  $2.7\times10^{-3} \pm 1\times10^{-5}$ & $5.8\times10^{-3}$ & $6.42\times10^3$\\
$1600$&$160 $  & $1.1\times10^{-4} \pm 9\times10^{-7}$ & $4.8\times10^{-4}$ & $1.29\times10^4$\\
$3200$&$320 $ & $4.3\times10^{-7} \pm 3\times10^{-8}$ & $1.8\times10^{-5}$ & $2.58\times10^4$ \\
\hline\end{tabular}
 \end{small}
\end{center}
\label{tab:BiasGARCH0.1}
\end{table}

\begin{table}
\caption{Estimation of \(\lim_{i\rightarrow\infty}\Pr(\sigma_{i}^{2}> z)\) in a GARCH volatility model with  \(b(k)=b'(k)=\lfloor k/10\rfloor\). }
\begin{small}
\begin{tabular}{rrllllcr}\hline
$k$ &  burn-in & Method  & $\mu$  & Std & RMSE & Cost& Cost $\times$ MSE \\ \hline
$50$&$5 $&LR & $0.1126$ & $1.8\times10^{-1}$ & $3.4\times10^{-1}$ & $5.00\times10^1$ & $5.8$\\
& &ULR & $0.398 \pm 0.003$ & $1.4\times10^0$ & $1.4\times10^0$ & $1.02\times10^2$ & $200$\\
& &SULR & $0.400 \pm 0.002$ & $1.2\times10^{-3}$ & $1.2\times10^{-3}$ & $1.01\times10^8$ & $140$\\
$200$&$20 $&LR & $0.3319$ & $2.1\times10^{-1}$ & $2.2\times10^{-1}$ & $2.00\times10^2$ & $9.5$\\
& &ULR & $0.3991 \pm 0.0008$ & $4.2\times10^{-1}$ & $4.2\times10^{-1}$ & $4.03\times10^2$ & $71$\\
& &SULR & $0.3993 \pm 0.0007$ & $3.8\times10^{-4}$ & $3.8\times10^{-4}$ & $3.96\times10^8$ & $58$\\
$800$&$80 $&LR & $0.3969$ & $1.2\times10^{-1}$ & $1.2\times10^{-1}$ & $8.00\times10^2$ & $11$\\
& &ULR & $0.3996 \pm 0.0002$ & $1.2\times10^{-1}$ & $1.2\times10^{-1}$ & $1.59\times10^3$ & $23$\\
& &SULR & $0.3996 \pm 0.0002$ & $1.2\times10^{-4}$ & $1.2\times10^{-4}$ & $1.59\times10^9$ & $23$\\
$3200$&$320 $&LR & $0.39963$ & $6.1\times10^{-2}$ & $6.1\times10^{-2}$ & $3.20\times10^3$ & $12$\\
& &ULR & $0.39970 \pm 0.0001$ & $6.1\times10^{-2}$ & $6.1\times10^{-2}$ & $6.34\times10^3$ & $23$\\
& &SULR & $0.39963 \pm 0.0001$ & $6.1\times10^{-5}$ & $6.1\times10^{-5}$ & $6.28\times10^9$ & $23$\\
\hline\end{tabular}
 \end{small}
\label{tab:GARCH0.1}
\end{table}

\begin{table}
\caption{Estimation of   \(\lim_{i\rightarrow\infty}\Pr(\sigma_{i}^{2}> z)\) in a GARCH volatility model with  \(b(k)=b'(k)=\lfloor k/2\rfloor\). }
\begin{small}
\begin{tabular}{rrllllcr}\hline
$k$ &  burn-in & Method   & $\mu$  & Std & RMSE & Cost& Cost $\times$ MSE \\ \hline
$50$&$25 $&LR & $0.1745$ & $2.8\times10^{-1}$ & $3.6\times10^{-1}$ & $5.00\times10^1$ & $6.3$\\
& &ULR & $0.398 \pm 0.002$ & $1.2\times10^0$ & $1.2\times10^0$ & $1.02\times10^2$ & $156$\\
& &SULR & $0.399 \pm 0.002$ & $1.1\times10^{-3}$ & $1.1\times10^{-3}$ & $1.01\times10^8$ & $120$\\
$200$ & $100 $ &LR & $0.3892$ & $2.7\times10^{-1}$ & $2.7\times10^{-1}$ & $2.00\times10^2$ & $15$\\
& &ULR & $0.3995 \pm 0.0006$ & $2.8\times10^{-1}$ & $2.8\times10^{-1}$ & $4.03\times10^2$ & $33$\\
& &SULR & $0.3993 \pm 0.0006$ & $2.8\times10^{-4}$ & $2.8\times10^{-4}$ & $3.96\times10^8$ & $32$\\
$800$ & $400 $ &LR & $0.3995$ & $1.6\times10^{-1}$ & $1.6\times10^{-1}$ & $8.00\times10^2$ & $20$\\
& &ULR & $0.3996 \pm 0.0003$ & $1.6\times10^{-1}$ & $1.6\times10^{-1}$ & $1.59\times10^3$ & $40$\\
& &SULR & $0.3995 \pm 0.0003$ & $1.6\times10^{-4}$ & $1.6\times10^{-4}$ & $1.59\times10^9$ & $40$\\
$3200$ & $1600 $ &LR & $0.39959$ & $8.1\times10^{-2}$ & $8.1\times10^{-2}$ & $3.20\times10^3$ & $21$\\
& &ULR & $0.39966 \pm 0.0002$ & $8.1\times10^{-2}$ & $8.1\times10^{-2}$ & $6.34\times10^3$ & $42$\\
& &SULR & $0.39959 \pm 0.0002$ & $8.1\times10^{-5}$ & $8.1\times10^{-5}$ & $6.28\times10^9$ & $42$\\
\hline\end{tabular}
 \end{small}
\label{tab:GARCH0.5}
\end{table}
\subsection{\(M/H_{k}/1\) queue}Consider a single-server queue with Poisson arrivals at rate \(\lambda=0.75\), where the service time \(V_{n}\) for the \(n\)-th customer      has an hyperexponential distribution with \(\Pr(V_{n}\geq z)=pe^{-2pz}+(1-p)e^{-2(1-p)z}\) for \(z\geq0\), with \(p=0.8875 \). The service-time parameters are taken from \cite{whitt1991longRun}. Tables~\ref{tab:MH0.1} and~\ref{tab:MH0.5} estimate \(\lim_{i\rightarrow\infty}E(X_{i})\). The center panel of
  Fig.~\ref{fig:GARCHMH1GG1} shows that the bias decays at a rate faster  than that of the standard deviation and, in Tables~\ref{tab:MH0.1} and \ref{tab:MH0.5}, the total running time of ULR and of SULR is about twice that of LR. Once again, LR exhibits a strong bias for small values of  \(k\).  When \(k\) is sufficiently large so that the bias is small,  the product \(\text{Cost}\times\text{MSE}\) is twice as large for ULR and SULR as for LR.  Here again, the performance of ULR and of  SULR tends to increase with \(k\),    is higher in  Table~\ref{tab:MH0.1} than in Table~\ref{tab:MH0.5} for large  values of \(k\), while the reverse effect is true for small values of \(k\). 
By the Pollaczek--Khinchine formula,
\begin{displaymath}
\lim_{i\rightarrow\infty}E(X_{i})=\frac{\lambda E(S^{2}) }{2(1-\lambda E(S))}=\frac{\lambda }{4p(1-p)(1-\lambda )}\approx7.51174,
\end{displaymath}
which is consistent with the results in Tables~\ref{tab:MH0.1} and~\ref{tab:MH0.5}. \begin{table}
\caption{Absolute value of bias in estimating \(E(X_{i})\) in an \(M/H_{k}/1\) queue with burn-in period  \(b(k)=\lfloor k/10\rfloor\) and \(10^{6}\) independent replications.}
\begin{center}
\begin{small}
\begin{tabular}{rrllc}\hline
$k$ &  burn-in   & $95\%$ confidence interval  & Std  & Cost  \\ \hline
 $25$&$2 $  & $4.4 \pm 0.04$ & $1.9\times10^1$ & $2.01\times10^2$ \\
$50$&$5 $  & $3.2 \pm 0.02$ & $1.3\times10^1$ & $4.04\times10^2$ \\
$100$&$10 $  & $1.9 \pm 0.02$ & $8.6\times10^0$ & $8.06\times10^2$ \\
$200$&$20 $  & $0.95 \pm 0.01$ & $5.3\times10^0$ & $1.62\times10^3$ \\
$400$&$40 $ & $0.36 \pm 0.005$ & $2.8\times10^0$ & $3.25\times10^3$ \\
$800$&$80 $ & $0.10 \pm 0.002$ & $1.1\times10^0$ & $6.50\times10^3$ \\
$1600$&$160 $ & $0.017 \pm 0.0006$ & $3.2\times10^{-1}$ & $1.28\times10^4$ \\
$3200$&$320 $ & $0.0013 \pm 0.0001$ & $6.7\times10^{-2}$ & $2.62\times10^4$ \\
\hline\end{tabular}
 \end{small}
 \end{center}
\label{tab:BiasMH0.1}
\end{table}

\begin{table}
\caption{Estimation of \(E(X_{i})\) in an \(M/H_{k}/1\) queue with  \(b(k)=b'(k)=\lfloor k/10\rfloor\).}
\begin{small}
\begin{tabular}{rrllllcr}\hline
$k$ &  burn-in & Method  & $\mu$  & Std & RMSE & Cost& Cost $\times$ MSE \\ \hline
$50$&$5 $&LR & $4.35$ & $4.9\times10^0$ & $5.9\times10^0$ & $5.00\times10^1$ & $1.7\times10^3$\\
& &ULR & $7.53 \pm 0.07$ & $3.8\times10^1$ & $3.8\times10^1$ & $1.00\times10^2$ & $1.5\times10^5$\\
& &SULR & $7.53 \pm 0.07$ & $3.8\times10^{-2}$ & $3.8\times10^{-2}$ & $9.87\times10^7$ & $1.4\times10^5$\\
$200$&$20 $&LR & $6.547$ & $5.8\times10^0$ & $5.8\times10^0$ & $2.00\times10^2$ & $6.8\times10^3$\\
& &ULR & $7.51 \pm 0.03$ & $1.7\times10^1$ & $1.7\times10^1$ & $4.00\times10^2$ & $1.1\times10^5$\\
& &SULR & $7.50 \pm 0.03$ & $1.6\times10^{-2}$ & $1.6\times10^{-2}$ & $4.00\times10^8$ & $1.1\times10^5$\\
$800$&$80 $&LR & $7.412$ & $4.2\times10^0$ & $4.2\times10^0$ & $8.00\times10^2$ & $1.4\times10^4$\\
& &ULR & $7.503 \pm 0.01$ & $5.3\times10^0$ & $5.3\times10^0$ & $1.59\times10^3$ & $4.4\times10^4$\\
& &SULR & $7.516 \pm 0.01$ & $5.4\times10^{-3}$ & $5.4\times10^{-3}$ & $1.60\times10^9$ & $4.7\times10^4$\\
$3200$&$320 $&LR & $7.512$ & $2.3\times10^0$ & $2.3\times10^0$ & $3.20\times10^3$ & $1.6\times10^4$\\
& &ULR & $7.511 \pm 0.004$ & $2.3\times10^0$ & $2.3\times10^0$ & $6.29\times10^3$ & $3.2\times10^4$\\
& &SULR & $7.513 \pm 0.004$ & $2.3\times10^{-3}$ & $2.3\times10^{-3}$ & $6.70\times10^9$ & $3.4\times10^4$\\
\hline\end{tabular}
 \end{small}
\label{tab:MH0.1}
\end{table}

\begin{table}
\caption{Estimation of \(E(X_{i})\) in an \(M/H_{k}/1\) queue with  \(b(k)=b'(k)=\lfloor k/2\rfloor\).}
\begin{small}
\begin{tabular}{rrllllcr}\hline
$k$ &  burn-in & Method   & $\mu$  & Std & RMSE & Cost& Cost $\times$ MSE \\ \hline
$50$&$25 $&LR & $5.15$ & $6.5\times10^0$ & $6.9\times10^0$ & $5.00\times10^1$ & $2.4\times10^3$\\
& &ULR & $7.53 \pm 0.07$ & $3.5\times10^1$ & $3.5\times10^1$ & $1.00\times10^2$ & $1.2\times10^5$\\
& &SULR & $7.52 \pm 0.07$ & $3.5\times10^{-2}$ & $3.5\times10^{-2}$ & $9.87\times10^7$ & $1.2\times10^5$\\
$200$&$100 $&LR & $7.11$ & $7.6\times10^0$ & $7.6\times10^0$ & $2.00\times10^2$ & $1.2\times10^4$\\
& &ULR & $7.51 \pm 0.03$ & $1.4\times10^1$ & $1.4\times10^1$ & $4.00\times10^2$ & $7.7\times10^4$\\
& &SULR & $7.51 \pm 0.03$ & $1.4\times10^{-2}$ & $1.4\times10^{-2}$ & $4.00\times10^8$ & $7.4\times10^4$\\
$800$&$400 $&LR & $7.508$ & $5.5\times10^0$ & $5.5\times10^0$ & $8.00\times10^2$ & $2.4\times10^4$\\
& &ULR & $7.508 \pm 0.01$ & $5.6\times10^0$ & $5.6\times10^0$ & $1.59\times10^3$ & $5.0\times10^4$\\
& &SULR & $7.514 \pm 0.01$ & $5.7\times10^{-3}$ & $5.7\times10^{-3}$ & $1.60\times10^9$ & $5.2\times10^4$\\
$3200$&$1600 $&LR & $7.512$ & $3.0\times10^0$ & $3.0\times10^0$ & $3.20\times10^3$ & $2.9\times10^4$\\
& &ULR & $7.509 \pm 0.006$ & $3.0\times10^0$ & $3.0\times10^0$ & $6.29\times10^3$ & $5.6\times10^4$\\
& &SULR & $7.512 \pm 0.006$ & $3.0\times10^{-3}$ & $3.0\times10^{-3}$ & $6.70\times10^9$ & $6.0\times10^4$\\
\hline\end{tabular}
 \end{small}
\label{tab:MH0.5}
\end{table}
\subsection{\(GI/G/1\) queue}Assume that the  interarrival time \(D_{n}\) and service time \(V_{n}\) for the \(n\)-th customer have Pareto distributions with  \(\Pr(D_{n}\geq z)=(1+z)^{-7}\) and \(\Pr(V_{n}\geq z)=(1+z/\alpha)^{-7}\) for \(z\geq0\), with   \(\alpha=0.8\).
Tables~\ref{tab:GG0.1} and~\ref{tab:GG0.5} estimate \(\lim_{i\rightarrow\infty}\Pr(X_{i}>1)\).  The right panel  Fig.~\ref{fig:GARCHMH1GG1} shows that the bias decays at a rate faster  than that of the standard deviation, the total running time of ULR and of SULR is about twice that of LR, and the bias of LR is strong for small values of  \(k\).  When \(k\) is large enough so that  the bias is small,  the product \(\text{Cost}\times\text{MSE}\) is twice as large for ULR and SULR as for LR.  Here again, the performance of ULR and of  SULR tends to increase with \(k\), is higher in  Table~\ref{tab:GG0.1} than in Table~\ref{tab:GG0.5} for large  values of \(k\), while the reverse  is true for small values of \(k\).
\begin{table}
\caption{Absolute value of bias in estimating of \(\lim_{i\rightarrow\infty}\Pr(X_{i}>1)\) in a \(GI/G/1\) queue with burn-in period  \(b(k)=\lfloor k/10\rfloor\) and \(10^{6}\) independent replications.}
\begin{center}
\begin{small}
\begin{tabular}{rrlccc}\hline
$k$ &  burn-in  & $95\%$ confidence interval  & Std  & Cost \\ \hline
$25$&$2 $ & $2.1\times10^{-1} \pm 1\times10^{-3}$ & $6.3\times10^{-1}$ & $2.11\times10^2$\\
$50$&$5 $ & $1.3\times10^{-1} \pm 8\times10^{-4}$ & $3.9\times10^{-1}$ & $4.03\times10^2$ \\
$100$&$10 $ & $7.0\times10^{-2} \pm 4\times10^{-4}$ & $2.3\times10^{-1}$ & $8.12\times10^2$ \\
$200$&$20 $ & $3.0\times10^{-2} \pm 2\times10^{-4}$ & $1.2\times10^{-1}$ & $1.60\times10^3$ \\
$400$&$40 $ & $9.9\times10^{-3} \pm 1\times10^{-4}$ & $5.1\times10^{-2}$ & $3.20\times10^3$ \\
$800$&$80 $ & $2.3\times10^{-3} \pm 3\times10^{-5}$ & $1.8\times10^{-2}$ & $6.47\times10^3$\\
$1600$&$160 $  & $3.4 \times10^{-4}\pm 1\times10^{-5}$ & $4.9\times10^{-3}$ & $1.31\times10^4$ \\
$3200$&$320 $ & $2.1\times10^{-5} \pm 2\times10^{-6}$ & $8.7\times10^{-4}$ & $2.58\times10^4$\\

\hline\end{tabular}
 \end{small}
 \end{center}
\label{tab:BiasGG0.1}
\end{table}

\begin{table}
\caption{Estimation of \(\lim_{i\rightarrow\infty}\Pr(X_{i}>1)\) in a \(GI/G/1\) queue with \(b(k)=b'(k)=\lfloor k/10\rfloor\).}
\begin{small}
\begin{tabular}{rrllllrr}\hline
$k$ &  burn-in & Method  & $\mu$ & Std & RMSE & Cost& Cost $\times$ MSE\\ \hline
$50$&$5 $&LR & $0.2004$ & $2.4\times10^{-1}$ & $2.7\times10^{-1}$ & $5.00\times10^1$ & $3.8$\\
& &ULR & $0.33 \pm 0.002$ & $1.2\times10^0$ & $1.2\times10^0$ & $1.00\times10^2$ & $135$\\
& &SULR & $0.334 \pm 0.002$ & $1.1\times10^{-3}$ & $1.1\times10^{-3}$ & $1.00\times10^8$ & $131$\\
$200$&$20 $&LR & $0.302$ & $2.0\times10^{-1}$ & $2.1\times10^{-1}$ & $2.00\times10^2$ & $8.4$\\
& &ULR & $0.3327 \pm 0.0008$ & $4.0\times10^{-1}$ & $4.0\times10^{-1}$ & $3.93\times10^2$ & $64$\\
& &SULR & $0.3323 \pm 0.0008$ & $4.0\times10^{-4}$ & $4.0\times10^{-4}$ & $3.97\times10^8$ & $63$\\
$800$&$80 $&LR & $0.3298$ & $1.2\times10^{-1}$ & $1.2\times10^{-1}$ & $8.00\times10^2$ & $12$\\
& &ULR & $0.3321 \pm 0.0003$ & $1.3\times10^{-1}$ & $1.3\times10^{-1}$ & $1.64\times10^3$ & $29$\\
& &SULR & $0.3321 \pm 0.0003$ & $1.3\times10^{-4}$ & $1.3\times10^{-4}$ & $1.63\times10^9$ & $28$\\
$3200$&$320 $&LR & $0.33222$ & $6.3\times10^{-2}$ & $6.3\times10^{-2}$ & $3.20\times10^3$ & $13$\\
& &ULR & $0.33221 \pm 0.0001$ & $6.3\times10^{-2}$ & $6.3\times10^{-2}$ & $6.35\times10^3$ & $25$\\
& &SULR & $0.33224 \pm 0.0001$ & $6.3\times10^{-5}$ & $6.3\times10^{-5}$ & $6.33\times10^9$ & $25$\\
\hline\end{tabular}
\end{small}
\label{tab:GG0.1}
\end{table}

\begin{table}
\caption{Estimation of \(\lim_{i\rightarrow\infty}\Pr(X_{i}>1)\) in a \(GI/G/1\) queue with \(b(k)=b'(k)=\lfloor k/2\rfloor\).}
\begin{small}
\begin{tabular}{rrllllrr}\hline
$k$ &  burn-in & Method & $\mu$ & Std & RMSE & Cost& Cost $\times$ MSE\\ \hline $50$&$25 $&LR & $0.2479$ & $3.2\times10^{-1}$ & $3.3\times10^{-1}$ & $5.00\times10^1$ & $5.4$\\
& &ULR & $0.330 \pm 0.002$ & $1.1\times10^0$ & $1.1\times10^0$ & $1.00\times10^2$ & $111$\\
& &SULR & $0.333 \pm 0.002$ & $1.1\times10^{-3}$ & $1.1\times10^{-3}$ & $1.00\times10^8$ & $112$\\
$200$&$100 $&LR & $0.3228$ & $2.6\times10^{-1}$ & $2.6\times10^{-1}$ & $2.00\times10^2$ & $14$\\
& &ULR & $0.3325 \pm 0.0007$ & $3.5\times10^{-1}$ & $3.5\times10^{-1}$ & $3.93\times10^2$ & $48$\\
& &SULR & $0.3322 \pm 0.0007$ & $3.5\times10^{-4}$ & $3.5\times10^{-4}$ & $3.97\times10^8$ & $49$\\
$800$&$400 $&LR & $0.3321$ & $1.6\times10^{-1}$ & $1.6\times10^{-1}$ & $8.00\times10^2$ & $21$\\
& &ULR & $0.3319 \pm 0.0003$ & $1.6\times10^{-1}$ & $1.6\times10^{-1}$ & $1.64\times10^3$ & $42$\\
& &SULR & $0.3321 \pm 0.0003$ & $1.6\times10^{-4}$ & $1.6\times10^{-4}$ & $1.63\times10^9$ & $42$\\
$3200$&$1600 $&LR & $0.33234$ & $8.4\times10^{-2}$ & $8.4\times10^{-2}$ & $3.20\times10^3$ & $23$\\
& &ULR & $0.33227 \pm 0.0002$ & $8.4\times10^{-2}$ & $8.4\times10^{-2}$ & $6.35\times10^3$ & $45$\\
& &SULR & $0.33234 \pm 0.0002$ & $8.4\times10^{-5}$ & $8.4\times10^{-5}$ & $6.33\times10^9$ & $45$\\
\hline\end{tabular}
 \end{small}
\label{tab:GG0.5}
\end{table}
\section{Conclusion}
Under a coupling assumption, we have established bounds on the bias, variance and mean square error of standard time-average estimators, and shown the sharpness of the variance and mean square error bounds. We have  built an unbiased RMLMC estimator for the bias of a conventional time-average estimator. Combining this unbiased estimator with a conventional time-average estimator yields an unbiased estimator  \(\hat f_{k}\)  of \(\mu\).  Both unbiased estimators are square-integrable  and have finite expected running time. Under certain conditions, they can be built  without any precomputations. For a suitable choice of parameters, we have shown that   \(\hat f_{k}\) is  asymptotically at least as efficient as \(f_{k}\), up to a multiplicative factor arbitrarily close to \(1\). We have also constructed an efficient stratified version \(\tilde f_{k,n}\) of    \(\hat f_{k}\). Building more refined stratified versions of    \(\hat f_{k}\),  such as those in \cite{Vihola2018}, is left for future research. Our approach permits to estimate
the bias of \(f_{k}\) and to determine the number of time-steps needed to substantially reduce it. It  can be implemented in a parallelized fashion and allows the robust construction of confidence intervals for \(\mu\). We have provided examples in  volatility forecasting,  queues, and  the simulation of high-dimensional Gaussian vectors where our approach is provably efficient, even when \(f\) is discontinuous. Our numerical experiments are consistent with our theoretical findings. In our experiments, the value  of \(q\) is fixed
and   \(f_{k}\) is about twice as efficient as  \(\hat f_{k}\) and  \(\tilde f_{k,n}\)  when \(k\) is sufficiently large.  As per the discussion following Theorem~\ref{th:maintheta}, for large values of \(k\), the performance of     \(\hat f_{k}\) and of    \(\tilde f_{k,n}\)   should increase if  \(q\) decreases. In practice, though, determining the optimal value of \(q\) for a given \(k\) may require a large amount of pre-computations. For     \(\tilde f_{k,n}\), for instance, it can be shown that the optimal value of \(q\) depends on the variance of \(Z_k\), that is not easy to estimate accurately for large values of \(k\).
%\section*{Acknowledgments}
%The author thanks Francis Bach for fruitful discussions. 
\appendix
\section{Proof of Proposition~\ref{pr:contractive1}}
We  show by induction on \(i\) that, for \(i,m\geq0\), \begin{equation}
\label{eq:inductiond2}E(\rho^{2}(X_{i},X_{i,m}))\leq\kappa'\eta^{i}.
\end{equation}As \(X_{0,m}\sim X_{m}\), \eqref{eq:inductiond2}  holds for \(i=0\). Assume now that  \eqref{eq:inductiond2} holds for \(i\). It follows from the definitions of \(G_{i}\) and of \(X_{i,m}\) that\begin{equation*}
X_{i+1,m}=g(X_{i,m},U_{i}).
\end{equation*}
Together with \eqref{eq:recXi} and \eqref{eq:distanceContractive}, this implies that\begin{equation*}
E(\rho^{2}(X_{i+1},X_{i+1,m}))\leq\eta E(\rho^{2}(X_{i},X_{i,m})).
\end{equation*}
Thus  \eqref{eq:inductiond2} holds for \(i+1\). Combining~\eqref{eq:defkappaLips} and \eqref{eq:inductiond2} shows that \(E((f(X_{i,m})-f(X_{i}))^{2})\leq\kappa^{2}\kappa'^{\gamma}\eta^{\gamma i}\) for \(i,m\geq0\). This concludes the proof.\qed
\section{Proof of Lemma~\ref{le:StdSum}}
We first prove the following.\begin{proposition}
\label{pr:shiftedPair}For  \(n,m,m'\in\mathbb{Z}\) with \(n\geq-m\) and \( n\geq -m'\), we have\begin{displaymath}
E((f(X_{n,m'})-f(X_{n,m}))^{2})=E((f(X_{n+m,m'-m})-f(X_{n+m}))^{2}).
\end{displaymath} \end{proposition}
\begin{proof}
 We have \(X_{n,m}=G_{n+m}(X_{0};U_{-m},\dots,U_{n-1})\), and  \(X_{n,m'}=G_{n+m'}(X_{0};U_{-m'},\dots,U_{n-1})\). Also,  \(X_{n+m}=G_{n+m}(X_{0};U_{0},\dots,U_{n+m-1})\), and \(X_{n+m,m'-m}=G_{n+m'}(X_{0};U_{m-m'},\dots,U_{n+m-1})\). As\begin{displaymath}
((U_{-m},\dots,U_{n-1}),(U_{-m'},\dots,U_{n-1}))\sim((U_{0},\dots,U_{n+m-1}),(U_{m-m'},\dots,U_{n+m-1})),  \end{displaymath} the pair \((X_{n,m},X_{n,m'})\) has the same distribution as  \((X_{n+m},X_{n+m,m'-m})\).
This concludes the proof.\end{proof}

We now prove the lemma. For \(l\geq 0\), let \begin{displaymath}
\sigma_{l}:=2^{l/2}\left( \sqrt{\nu(0)}+(\sqrt{2}+1)\sum^{2^{l}}_{i=1} \sqrt{\nu(i)}(i^{-1/2}-2^{-l/2})\right).
\end{displaymath}In particular, we have \(\sigma_{0}=\sqrt{\nu(0)}\). We show by induction on \(l\) that \begin{equation}\label{eq:indSigmal}
\std(\sum ^{h+k-1}_{i=h}f(X_{i}))\le \sigma_{l} \text{ for }h\ge0\text{ and }\,0\leq k\leq2^{l}.
\end{equation} 
If \(k=0\), the summation in the left-hand side of    \eqref{eq:indSigmal} is null by convention, and     \eqref{eq:indSigmal} trivially holds.  Applying \eqref{eq:rhoDef} with \(i=0\) shows that, for \(m\ge 0\),
\begin{equation}\label{eq:upper-boundRho0}
E((f(X_{m})-f(X_{0}))^{2})\leq\nu(0).
\end{equation}Hence \eqref{eq:indSigmal} holds for \(l=0\). Assume now that  \eqref{eq:indSigmal} holds for \(l\). We show  that it holds for \(l+1\). Fix non-negative integers \(k\) and \(h\), with \(0\le k\leq2^{l+1}\). If \(k\leq1\) then   \eqref{eq:indSigmal} holds for \(l+1\) as a consequence of \eqref{eq:upper-boundRho0}. Assume now that \(k>1\) and let   \(j=\lfloor k/2\rfloor\).    For \(i\geq0\), let \(X'_{i}=X_{h+j+i-1,1-h-j}\). By~\eqref{eq:recXim}
and the remark that follows it,  the sequences \((X_{i},0\leq i\leq k-j)\) and \((X'_{i},0\leq i\leq k-j)\)  have the same distribution. Set
\(V_{1}=\sum^{{h+j-1}}_{i=h}f(X_{i})\), 
\(V_{2}=\sum^{{h+k-1}}_{i=h+j}f(X_{i})\), 
  and \(V'_{2}:=\sum ^{k-j}_{i=1}f(X'_{i})\). Then \begin{eqnarray}\label{eq:stdSummation}
\std(\sum^{{h+k-1}}_{i=h}f(X_{i}))&=&\std(V_{1}+V_{2})\nonumber\\
&\le&\std(V_{1}+V'_{2})+\std(V_{2}-V'_{2}).
\end{eqnarray}
The second equation follows from the sub-linearity of the standard deviation,
 i.e., \(\std(V+V')\le\std(V)+\std(V')\) for any square-integrable random variables \(V\) and \(V'\). Note that \(V'_{2}\)  has the same distribution as \(\sum ^{k-j}_{i=1}f(X_{i})\).   As \(j\) and \(k-j\) are upper-bounded by \(2^{l}\), the induction hypothesis implies that \(\std(V_{1})\) and  \(\std(V'_{2})\) are both upper-bounded by \(\sigma_{l}\).  Furthermore,
by construction,  \(V_{1}\) (resp. \(V'_{2}\))
 is a deterministic measurable function of \((U_{0},\dots,U_{h+j-2})\) (resp.
 \((U_{h+j-1},\dots,U_{h+k-2})\)). Thus, \(V_{1}\) and \(V'_{2}\)  are independent. Hence
\begin{eqnarray}\label{eq:stdSum}
\var(V_{1}+V'_{2})&=&\var(V_{1})+\var(V'_{2})\nonumber\\
&\le&2{\sigma_l}^2.
\end{eqnarray}
Let  \(i\in[1, k  -j]\). Applying Proposition~\ref{pr:shiftedPair} with  \(m'=0\), \(n=h+j+i-1\) and \(m=1-h-j\) shows that
\begin{displaymath}
E((f(X_{h+j+i-1})-f(X'_{i}))^{2})=E((f(X_{i,h+j-1})-f(X_i))^{2}).
\end{displaymath} Since \(j>0\), together with~\eqref{eq:rhoDef}, this implies that       \(E((f(X_{h+j+i-1})-f(X'_{i}))^{2})\leq \nu(i)\). Thus 
\begin{eqnarray}\label{eq:stdDiff}
\std(V_{2}-V'_{2})&=&\std\left(\sum^{k-j}_{i=1}(f(X_{h+j+i-1})-f(X'_{i}))\right)\nonumber\\
&\le&\sum^{k-j}_{i=1}\sqrt{\nu(i)},
\end{eqnarray} where the second equation follows from the sub-linearity of the standard deviation.
Combining \eqref{eq:stdSummation}, \eqref{eq:stdSum} and \eqref{eq:stdDiff} yields \begin{equation*}
\std(\sum ^{h+k-1}_{i=h}f(X_{i}))\le \sqrt{2}\sigma_{l}+\sum^{2^{l}}_{i=1}\sqrt{\nu(i)}.
\end{equation*}
By the definition of \(\sigma_{l}\),
\begin{eqnarray*}
 \sqrt{2}\sigma_{l}+\sum^{2^{l}}_{i=1}\sqrt{\nu(i)} &=&2^{(l+1)/2}\left( \sqrt{\nu(0)}+(\sqrt{2}+1)\sum^{2^{l}}_{i=1} \sqrt{\nu(i)}(i^{-1/2}-2^{-l/2})\right)+\sum^{2^{l}}_{i=1}\sqrt{\nu(i)}
\\&=&2^{(l+1)/2}\left( \sqrt{\nu(0)}+(\sqrt{2}+1)\sum^{2^{l}}_{i=1} \sqrt{\nu(i)}(i^{-1/2}-2^{-(l+1)/2})\right)
\\&\leq&\sigma_{l+1},
\end{eqnarray*}
where the second equation follows from standard calculations. Thus, the induction hypothesis holds for \(l+1\). 

Given  \(h\geq0\) and  \(k\geq1\), let \(l:=\lceil\log_{2}(k)\rceil\). By \eqref{eq:indSigmal},
\begin{eqnarray*}
\std(\sum ^{h+k-1}_{i=h}f(X_{i}))
&\leq&2^{l/2}\left( \sqrt{\nu(0)}+(\sqrt{2}+1)\sum^{2^{l}}_{i=1} \sqrt{\frac{\nu(i)}{{i}}}\right),\\
&\leq&\sqrt{2k}\left( \sqrt{\nu(0)}+(2+\sqrt{2})\sum^{2^{l}}_{i=1} \sqrt{\frac{\nu(i)}{{i+1}}}\right),\\
&\leq&2(\sqrt{2}+1)\sqrt{k}\sum^{2^{l}}_{i=0} \sqrt{\frac{\nu(i)}{{i+1}}},\\
\end{eqnarray*}      
where the second equation follows from the inequalities \(l\le\log_{2}(k)+1\) and \(2i\ge i+1\) for \(i\geq1\). As \(2(\sqrt{2}+1)\leq 5\), this completes the proof.\qed 
\section{Proof of Theorem~\ref{th:biasedLongRun}}
The following proposition gives a bound on the tail of the sequence \(\nu\).\begin{proposition}\label{pr:2kk'}For non-negative integers \(h,h'\) with \(2h\leq h'\),
we have \begin{equation*}
\sum^{h'}_{i=2h}\nu(i)\leq (\overline{\nu}(h))^{2}.
\end{equation*}  
\end{proposition}
\begin{proof}
We have
\begin{eqnarray*}\left(\sum_{i=h}^{h'}\sqrt{\frac{\nu(i)}{i+1}}\right)^{2}&=&\sum_{i=h}^{h'}\frac{\nu(i)}{i+1}+2\sum_{i=h}^{h'}\sqrt{\frac{\nu(i)}{i+1}}\left(\sum_{j=h}^{i-1}\sqrt{\frac{\nu(j)}{j+1}}\right)
\\&\ge&\sum_{i=h}^{h'}(2i-2h+1)\frac{\nu(i)}{i+1}\\
&\ge&\sum^{h'}_{i=2h}\nu(i).
\end{eqnarray*}
The second equation follows by observing that \(\nu(i)/(i+1)\le\nu(j)/(j+1)\) for \(0\leq j<i\), which implies that \(\sum_{j=h}^{i-1}\sqrt{{\nu(j)}/(j+1)}\ge(i-h)\sqrt{\nu(i)/(i+1)}\). 
\end{proof}

We now prove Theorem~\ref{th:biasedLongRun}. By Proposition~\ref{pr:2kk'},   \(\sum^{\infty}_{i=0}\nu(i)<\infty\). Hence \(\nu(i)\) goes to \(0\) as \(i\) goes to infinity. Since \((E(V))^{2}\leq E(V^{2})\) for any square-integrable random variable \(V\),  it follows from \eqref{eq:rhoDef} that, for \(i,m\geq0\),\begin{equation*}
|E(f(X_{i,m})-f(X_{i}))|\leq\sqrt{\nu(i)}.
\end{equation*}As  \(X_{i,m}\sim X_{i+m}\), we have \(E(f(X_{i,m}))=E(f(X_{i+m}))\). Therefore, \begin{equation}\label{eq:im}
|E(f(X_{i+m})-f(X_{i}))|\leq\sqrt{\nu(i)}.
\end{equation} Thus \((E(f(X_{h})),h\ge0)\) is a Cauchy sequence and  has a finite limit \(\mu\) as \(h\) goes to infinity.   
Letting \(m\) go to infinity in \eqref{eq:im}  implies \eqref{eq:thBias}.

By convexity of the square norm function,\begin{eqnarray*}(E(\frac{1}{k}\sum^{h+k-1}_{i=h}f(X_{i}))-\mu)^{2}&\leq& \frac{1}{k}\sum^{h+k-1}_{i=h}(E(f(X_{i}))-\mu)^{2},\\
&\le& \frac{1}{k}\sum^{h+k-1}_{i=h}\nu(i)\\
&\le& \frac{1}{k}\sum^{h+k-1}_{i=2\lfloor h/2\rfloor}\nu(i)\\&\le&\frac{1}{k}(\overline{\nu}(\lfloor h/2\rfloor))^{2}.
\end{eqnarray*} The third equation follows from the inequality \(2\lfloor h/2\rfloor\le h\), and the last one from Proposition~\ref{pr:2kk'}. This implies~\eqref{eq:thLongRunBias}.

As \(\overline{\nu}(\lfloor h/2\rfloor)\leq \overline{\nu}(0)\) for \(h\geq0\), it follows from~\eqref{eq:thLongRunBias} that, for \(k>0,\) \begin{equation}\label{eq:biasSimple}
(E(\frac{1}{k}\sum^{h+k-1}_{i=h}f(X_{i}))-\mu)^{2}\le\frac{(\overline{\nu}(0))^{2}}{k}.
\end{equation}Since the mean square error is related to the bias and standard deviation via the equation\begin{displaymath}
E((V-\mu)^{2})=(E(V)-\mu)^{2}+(\std(V))^{2},
\end{displaymath}for any square-integrable random variable \(V\), \eqref{eq:thMSE} follows by combining~\eqref{eq:biasSimple} with Lemma~\ref{le:StdSum}.\qed
\section{Proof of Proposition~\ref{pr:UpperBoundDelta}}
We first prove the following propositions.
\begin{proposition}\label{pr:MomentInequality}For \(u,v>0\) and \(\delta\geq0\), we have\begin{equation*}
(u+v)^{\delta}\leq2^{\delta}(u^{\delta}+v^{\delta}).
\end{equation*} 
\end{proposition}
\begin{proof}
Assume without loss of generality that \(u\leq v\). Then\begin{eqnarray*}(u+v)^{\delta}&\le&(2v)^{\delta}\\
&\le&2^{\delta}(u^{\delta}+v^{\delta}),\end{eqnarray*} as desired.    
\end{proof}
Proposition~\ref{pr:exponentialDecay} gives bounds on \(\overline{\omega}\) under an exponential decay assumption on \(\omega\). Up to a polylogarithmic factor, the bound on  \(\overline{\omega}(0)\) is inversely proportional to \(\sqrt{\xi}\), where \(\xi\) is the decay rate of \(\omega\). The bound on \(\overline{\omega}(j)\) is exponentially decaying with decay rate \(\xi/2\).  
\begin{proposition}\label{pr:exponentialDecay}Let \((\omega(i),i\geq0)\) be  a non-negative sequence such that \(\omega(i)\le c\ln^{\delta}(i+2) e^{-\xi i}\) for \(i\geq0\), where  \(c\) and \(\xi\) are positive constants, with \(\xi\leq1\) and \(\delta\ge0\). Then \begin{equation}\label{eq:S0upperBoundGen}
\overline{\omega}(0)\leq \sqrt{\frac{ec}{\xi}}\left(2^{\delta+1/2}\Gamma(\frac{\delta+1}{2})+2^{\delta/2}\sqrt{2\pi}\ln^{\delta/2}(\frac{2}{\xi})\right),
\end{equation}and, for \(j\geq 0\),
\begin{equation}\label{eq:SjUpperBoundGeom}
\overline{\omega}(j)\le5\sqrt{c}{\left(\frac{\delta}{e}\right)}^{\delta/2}\frac{e^{-\xi j/2}}{\xi},
\end{equation}
where \(0^{0}=1\) by convention.
\end{proposition}
\begin{proof} By replacing \(\omega\) with \((1/c)\omega\), it can be assumed without loss of generality that \(c=1\). We have
\begin{eqnarray*}
\overline{\omega}(0)&\leq&
\sum^{\infty}_{i=0}\frac{\ln^{\delta/2}(i+2)e^{-\xi i/2}}{\sqrt{i+1}}
 \\&\le&\sqrt{e}\int^{\infty}_{0}\frac{\ln^{\delta/2}(x+2)e^{-\xi x/2}}{\sqrt{x}}dx
 \\&=& \frac{2\sqrt{e}}{\sqrt{\xi}}\int^{\infty}_{0}\ln^{\delta/2}(y^{2}/\xi+2)e^{-y^{2}/2}dy, \end{eqnarray*}
where the second equation follows from the inequality \(e^{-\xi i/2}\leq \sqrt{e}e^{-\xi x/2}\) for \(x\in[i,i+1]\), and the third equation follows from the change of variables \(y=\sqrt{\xi x}\).
On the other hand, for \(y>0\), \begin{eqnarray*}\ln^{\delta/2}(y^{2}/\xi+2)
&\le&(\ln(y^{2}+1)+\ln(\frac{2}{\xi}))^{\delta/2}\\&\le&2^{\delta/2}(\ln^{\delta/2}(y^{2}+1)+\ln^{\delta/2}(\frac{2}{\xi}))\\
&\le&2^{\delta/2}(y^{\delta}+\ln^{\delta/2}(\frac{2}{\xi})),
\end{eqnarray*}where the second equation follows from Proposition~\ref{pr:MomentInequality}, and the last one from the inequality \(\ln(1+z)\leq z\) for \(z\geq0\). Thus \begin{eqnarray*}
\int^{\infty}_{0}\ln^{\delta/2}(y^{2}/\xi+2)e^{-y^{2}/2}dy&\le&2^{\delta/2}\int^{\infty}_{0}(y^{\delta}+\ln^{\delta/2}(\frac{2}{\xi}))e^{-y^{2}/2}dy\\&=&2^{\delta-1/2}\Gamma(\frac{\delta+1}{2})+2^{\delta/2}\ln^{\delta/2}(\frac{2}{\xi})\sqrt{\frac{\pi}{2}}.
\end{eqnarray*}This implies~\eqref{eq:S0upperBoundGen}. 

We now prove~\eqref{eq:SjUpperBoundGeom}. 
For \(j\geq 0\), we have
 \begin{eqnarray*}
 \overline{\omega}(j)&\le&
 \sum^{\infty}_{i=j}\ln^{\delta/2}(i+2)\frac{e^{-\xi i/2}}{\sqrt{i+1}}
 \\&\le&\sqrt{2}{\left(\frac{\delta}{e}\right)}^{\delta/2}\sum^{\infty}_{i=j}e^{-\xi i/2}
 \\&=&\sqrt{2}{\left(\frac{\delta}{e}\right)}^{\delta/2}\frac{e^{-\xi j/2}}{1-e^{-\xi/2}}
 \\&\le&2\sqrt{2e}{\left(\frac{\delta}{e}\right)}^{\delta/2}\frac{e^{-\xi j/2}}{\xi}, \end{eqnarray*}
 where the second equation follows from the inequality \(\ln^{\delta}(x)/x\leq(\delta/e)^{\delta}\) for \(x>1\), which implies that  \(\ln^{\delta}(x+1)/x\leq2(\delta/e)^{\delta}\) for \(x\geq1\),  and the last equation follows from the inequality \(1-e^{-x}\geq x/\sqrt{e}\) for \(x\in[0,1/2]\).
 \end{proof}

Assuming  a bound on   \(\omega\)  with an exponential decay rate  \(\xi\) combined with  an additional decay assumption, Proposition~\ref{pr:exponentialGeomDecay} shows that \(\overline{\omega}(0)\) is bounded by a polylogarithmic function of \(\xi\).
\begin{proposition}\label{pr:exponentialGeomDecay}Let \((\omega(i),i\geq0)\) be  a non-negative sequence such that,  for \(i\geq0\), \(\omega(i)\le c\ln^{\delta}(i+2)\min( e^{-\xi i},1/(i+1))\), where  \(\delta\geq0\) and  \(c\) and \(\xi\) are positive constants with \(\xi\leq1\). Then\begin{displaymath}
\overline{\omega}(0)\leq7\sqrt{ c}\ln^{\delta/2+1}\left(\frac{{\delta}^2+4}{\xi^{2}}\right).
\end{displaymath}
\end{proposition}
\begin{proof}By replacing \(\omega\) with \((1/c)\omega\), it can be assumed without loss of generality that \(c=1\). Set \(j=\lceil({\delta}^2+2)\xi^{-2}\rceil\). We have
\begin{equation*}
\overline{\omega}(0)=\sum^{j-1}_{i=0}\sqrt{\frac{\omega(i)}{i+1}}+\overline{\omega}(j).
\end{equation*}
Since \(\xi^{-1}\ge\ln(\xi^{-1})\), we have \(\xi j\geq{\delta}^2+2\ln(\xi^{-1})\). Hence \(e^{-\xi j/2}\le e^{-\delta^{2}/2}\xi\).  Proposition~\ref{pr:exponentialDecay} and the inequality \(\delta \leq e^{\delta}\) show  that \(\overline{\omega}(j)\le5\). On the other hand, 
\begin{eqnarray*}\sum^{j-1}_{i=0}\sqrt{\frac{\omega(i)}{i+1}}&\le&\ln^{\delta/2}(j+1)\sum^{j}_{i=1}\frac{1}{i}\\
&\le&(1+\ln(j))\ln^{\delta/2}(j+1).\end{eqnarray*}
Hence\begin{eqnarray*}
\overline{\omega}(0)&\leq&5+(1+\ln(j))\ln^{\delta/2}(j+1)\\
&\le&7\ln^{\delta/2+1}(j+1),
\end{eqnarray*}   where the second equation follows from the inequality \(1\leq\ln(j+1)\). As \(j+1\le({\delta}^2+4)\xi^{-2}\), this concludes the proof.
\end{proof}
We now prove Proposition~\ref{pr:UpperBoundDelta}. The sequence \(\nu\) satisfies the conditions of Proposition~\ref{pr:exponentialDecay} with \(\delta=0\). As  \(\Gamma(1/2)=\sqrt{\pi}\), \eqref{eq:S0upperBoundGen} shows that  \begin{equation*}
\overline{\nu}(0)\leq2\sqrt{\frac{2\pi e c}{\xi}}.
\end{equation*}
This implies \eqref{eq:deltaBound}.
Similarly, \eqref{eq:deltaBoundGeom} follows by applying Proposition~\ref{pr:exponentialGeomDecay} with \(\delta=0\).\qed
\section{Proof of  Lemma~\ref{le:boundVarfk0}}
We first give an upper bound on the standard deviation of \(f_{k,0}-f_{k}\). \begin{lemma}\label{le:diffStdFn0Fn}For \(k\ge1\), 
\begin{equation*}
\std(f_{k,0}-f_{k})\le \frac{25\overline{\nu}(0)}{k}\sqrt{b'(k)-b(k)}.  
\end{equation*}\end{lemma} 
\begin{proof}
We have
\begin{eqnarray*}
f_{k,0}-f_{k} &=&\frac{\sum_{i=b'(k)}^{k-1}f(X_{i})}{k-b'(k)}-\frac{\sum_{i=b(k)}^{k-1}f(X_{i})}{k-b(k)}\\
&=&\frac{\sum_{i=b'(k)}^{k-1}f(X_{i})}{k-b'(k)}-\frac{\sum_{i=b'(k)}^{k-1}f(X_{i})}{k-b(k)}+\frac{\sum_{i=b'(k)}^{k-1}f(X_{i})}{k-b(k)}-\frac{\sum_{i=b(k)}^{k-1}f(X_{i})}{k-b(k)}\\
&=&\frac{b'(k)-b(k)}{(k-b(k))(k-b'(k))}\sum_{i=b'(k)}^{k-1}f(X_{i})-\frac{\sum_{i=b(k)}^{b'(k)-1}f(X_{i})}{k-b(k)}.
\end{eqnarray*}
By sub-linearity of the standard deviation and the inequalities \(b(k)\leq b'(k)\leq k/2\), this implies that\begin{equation*}
\std(f_{k,0}-f_{k})\le4\frac{b'(k)-b(k)}{k^{2}}\std(\sum_{i=b'(k)}^{k-1}f(X_{i}))+\frac{2}{k}\std(\sum_{i=b(k)}^{b'(k)-1}f(X_{i})).
\end{equation*}  Using Lemma~\ref{le:StdSum}, it follows that\begin{equation*}
\std(f_{k,0}-f_{k})\le20\overline{\nu}(0)\frac{b'(k)-b(k)}{k^{3/2}}+\frac{10\overline{\nu}(0)}{k}\sqrt{b'(k)-b(k)}.\end{equation*} As  \(  b'(k)-b(k)\le k/2\), this concludes the proof. \end{proof}
We now prove Lemma~\ref{le:boundVarfk0}. By Lemma~\ref{le:diffStdFn0Fn} and the sub-linearity of the standard deviation, \begin{eqnarray*}
\var(f_{k,0})&\le&\left(\frac{25\overline{\nu}(0)}{k}\sqrt{b'(k)-b(k)}+\std(f_{k})\right)^{2}\\
&=&\frac{625\overline{\nu}(0)^{2}}{k^{2}}(b'(k)-b(k))+\frac{50\overline{\nu}(0)}{k}\sqrt{b'(k)-b(k)}\std(f_{k})+\var(f_{k}).
\end{eqnarray*}
We bound the first term by noting that\begin{displaymath}
b'(k)-b(k)\leq \sqrt{\frac{k(b'(k)-b(k))}{2}},
\end{displaymath} and the second term using the relation
\begin{equation}\label{eq:StdFnBound}
\std(f_{k})\leq\frac{5\sqrt{2}\overline{\nu}(0)}{\sqrt{k}}, 
\end{equation} which is a consequence of Lemma~\ref{le:StdSum}.
The lemma  follows after some simplifications.\qed
\section{Proof of Lemma~\ref{le:generalPl}}
For \(l\geq0\), let  \begin{displaymath}
\nu_{k,l}:=\frac{\sum ^{k2^{l}-1}_{i=b'(k)+k(2^{l}-1)}\nu(i)}{k-b'(k)}.
\end{displaymath}
Lemma~\ref{le:boundOnRhonl} gives bounds on the \(\nu_{k,l}\)'s. 
\begin{lemma}\label{le:boundOnRhonl}
For \(l\geq0\),  \begin{equation}\label{eq:rhonlGen}
k\nu_{k,l}\leq 2(\overline{\nu}(\lfloor{b'(k)}/{2}\rfloor))^{2},
\end{equation}
and, for \(l\geq2\),
\begin{equation}\label{eq:rhonlLargel}
k\nu_{k,l}\le2^{3-l}(\overline{\nu}(k2^{l-2})-\overline{\nu}(k2^{l-1}))^{2}.
\end{equation}
\end{lemma}
\begin{proof}
Applying Proposition~\ref{pr:2kk'} with \(h=\lfloor b'(k)/2\rfloor\) and \(h'=k2^{l}-1\) shows that \((k-b'(k))\nu_{k,l}\leq(\overline{\nu}(\lfloor b'(k)/2\rfloor))^{2}\). Since  \(b'(k)\leq k/2\), this implies~\eqref{eq:rhonlGen}. 

We now prove~\eqref{eq:rhonlLargel}. Let \(l\ge2\). As \(k(2^{l}-1)\geq k2^{l-1}\), for any \(i\geq k(2^{l}-1)\), we have  \(\nu(i)\leq\nu(k2^{l-1}) \). Since \(\nu_{k,l}\) is the average value of   \(\nu(i)\), where \(i\) ranges in \([b'(k)+k(2^{l}-1), k2^{l}-1]\),  it follows that \(\nu_{k,l}\leq\nu(k2^{l-1}) \). Consequently,\begin{eqnarray*}\overline{\nu}(k2^{l-2})-\overline{\nu}(k2^{l-1})&=&\sum^{k2^{l-1}-1}_{i=k2^{l-2}}\sqrt{\frac{\nu(i)}{i+1}}\\
&\ge&k2^{l-2}\sqrt{\frac{\nu(k2^{l-1})}{k2^{l-1}}}\\
&=&\sqrt{\nu(k2^{l-1})k2^{l-3}}\\
&\ge&\sqrt{\nu_{k,l}k2^{l-3}},
\end{eqnarray*}
where the second equation follows from the inequality \(\nu(i)\geq\nu(k2^{l-1})\) for \(i\leq k2^{l-1}\). Hence~\eqref{eq:rhonlLargel}.
\end{proof}
For \(l\geq0\), set \(Y_{l}=f_{k,l+1}-f_{k,0}\), with \(Y_{-1}=0\). By~\eqref{eq:ZkDef}, \(Z_{k}=(Y_{N}-Y_{N-1})/p_{N}\). \begin{lemma}\label{le:upperBoundDiffY}
For \(l\geq0\), \begin{equation*}
E((Y_{l}-Y_{l-1})^{2})\leq\nu_{k,l}.
\end{equation*}
\end{lemma}
\begin{proof}
It follows from Proposition~\ref{pr:shiftedPair} and~\eqref{eq:rhoDef} that \(E((f(X_{i,m'})-f(X_{i,m}))^{2})\leq\nu(i+m)\)  for \(i\geq0\) and \(0\leq m\leq m'\). Consequently,         \(E((f(X_{i,k(2^{l+1}-1)})-f(X_{i,k(2^{l}-1)}))^{2})\leq\nu(i+k(2^{l}-1))\) for  \(i,l\geq0\). For \(l\geq0\), we have\begin{eqnarray*}
Y_{l}-Y_{l-1}&=&f_{k,l+1}-f_{k,l}\\
&=&\frac{\sum ^{k-1}_{i=b'(k)}f(X_{i,k(2^{l+1}-1)})-f(X_{i,k(2^{l}-1)})}{k-b'(k)}. \end{eqnarray*}By the Cauchy-Schwarz inequality, \begin{eqnarray*}
E((Y_{l}-Y_{l-1})^{2})&\le&\frac{\sum ^{k-1}_{i=b'(k)}E((f(X_{i,k(2^{l+1}-1)})-f(X_{i,k(2^{l}-1)}))^{2})}{k-b'(k)}\\ &\le&\frac{\sum ^{k-1}_{i=b'(k)}\nu(i+k(2^{l}-1))}{k-b'(k)}\\
&=&\nu_{k,l}.
\end{eqnarray*}
\end{proof}
We now prove Lemma~\ref{le:generalPl}. The expected cost of computing  \(f_{k,l+1}-f_{k,l}\) is at most  \(3k2^{l}\). Thus  \(T_{k}\le3k\sum^{\infty}_{l=0}2^{l}p_{l}\).
 For any \(i\geq0\), as \(l\) goes to infinity, \(f(X_{i,k(2^{l}-1)})\) converges to \(\mu\) by Theorem~\ref{th:biasedLongRun}. Hence, by the definitions of \(f_{k,l}\) and of \(Y_{l}\), as \(l\) goes to infinity, \(E(f_{k,l})\) converges to \(\mu\) and \(E(Y_l)\) converges to \(\mu':=\mu-E(f_{k,0})\).  
If the right-hand side of \eqref{eq:ZNSecondMomentGeneralCase} is infinite, then  \eqref{eq:ZNSecondMomentGeneralCase} clearly holds. Assume now that the right-hand side of \eqref{eq:ZNSecondMomentGeneralCase} is finite.   Combining  Lemmas~\ref{le:boundOnRhonl} and~\ref{le:upperBoundDiffY} shows that, for \(0\leq l\leq1\),  \begin{equation*}
kE((Y_{l}-Y_{l-1})^{2})\leq 2(\overline{\nu}(\lfloor{b'(k)}/{2}\rfloor))^{2},
\end{equation*}
and, for \(l\geq2\),
\begin{equation*}
kE((Y_{l}-Y_{l-1})^{2})\le2^{3-l}(\overline{\nu}(k2^{l-2})-\overline{\nu}(k2^{l-1}))^{2}.
\end{equation*}
Hence
\(k\sum^{\infty}_{l=0}{E((Y_{l}-Y_{l-1})^{2})}/{p_{l}}\) is upper-bounded by the right-hand side of \eqref{eq:ZNSecondMomentGeneralCase}. 
 Theorem~\ref{th:Glynn}   shows that \(Z_{k}\) is square-integrable, that \eqref{eq:ZNSecondMomentGeneralCase} holds, and that \(E(Z_{k})=\mu'\). Consequently, \(E(f_{k,0} + Z_k) = \mu\).
\qed
\section{Proof of Lemma~\ref{le:specialPl}}
Denote by \(M\) the right-hand side of \eqref{eq:ZNSecondMomentGeneralCase}. As    \(p_{1}\ge1/4\) and    \(p_{0}\ge1/2\), we have
\begin{eqnarray*}
M
&\le& 12(\overline{\nu}(\lfloor{b'(k)}/{2}\rfloor))^{2}+8\overline{\nu}(k)\sum^{\infty}_{l=2}  (\overline{\nu}(k2^{l-2})-\overline{\nu}(k2^{l-1}))\\
&=&12(\overline{\nu}(\lfloor{b'(k)}/{2}\rfloor))^{2}+8(\overline{\nu}(k))^{2}\\
&\le&20(\overline{\nu}(\lfloor{b'(k)}/{2}\rfloor))^{2}.
\end{eqnarray*}Hence \(M\) is finite. By Lemma~\ref{le:generalPl}, this implies \eqref{eq:ZNSecondMoment} and that  \(E(f_{k,0}+Z_{k})=\mu\).
By~\eqref{eq:pr2^lnudep},
\begin{eqnarray*}
\sum^{\infty}_{l=0}2^{l}p_{l}&=&p_{0}+2p_{1}+1\\&\le&2p_{0}+2p_{1}+1\\&\le&3.\end{eqnarray*}
Using Lemma~\ref{le:generalPl}, it follows that   \(T_{k}\leq9k\).  
\qed
\section{Motivation for \eqref{eq:defq}}\label{se:motivationforq}
We first show the following.
\begin{proposition}\label{pr:optimalq}
Let \(V_{1}\) and \(V_{2}\) be independent square integrable random variables with finite expected running times \(\tau_{1}\) and \(\tau_{2}\). Let \(Q\) be a binary random variable independent of \((V_{1},V_{2})\), with \(\Pr(Q=1)=q\), where \(q\in(0,1]\). Set \(V=V_{1}+q^{-1}QV_{2}\). Let \(\tau\) be the expected  time to simulate \(V\). Then\begin{equation}\label{eq:varVtauBound}
\var(V)\tau\le(\var(V_{1})+q^{-1}E({V_{2}}^{2}))(\tau_{1}+q\tau_{2}).
\end{equation} If  \(0<E({V_{2}}^{2})\tau_{1}\le\var(V_{1})\tau_{2}\) then the RHS of \eqref{eq:varVtauBound} is minimized when \begin{equation}\label{eq:optimalq}
q=\sqrt{\frac{E({V_{2}}^{2})\tau_{1}}{\var(V_{1})\tau_{2}}}.
\end{equation}
\end{proposition}
\begin{proof}
We have\begin{eqnarray*}
\var(QV_{2})&\le&E((Q{V_{2}})^{2})\\&=& qE({V_{2}}^{2}).
\end{eqnarray*}
Hence
\begin{eqnarray}\label{eq:varV}
\var(V)&=&\var(V_{1})+q^{-2}\var(QV_{2})\nonumber\\
&\le&\var(V_{1})+q^{-1}E({V_{2}}^{2}).
\end{eqnarray}
Simulating \(V\) requires to simulate \(V_{1}\) and, when \(Q=1\), to simulate \(V_{2}\). Thus \(\tau=\tau_{1}+q\tau_{2}\). A standard calculation implies \eqref{eq:optimalq}.\end{proof}By Lemma~\ref{le:generalPl} and \eqref{eq:pr2^lnudep}, the expected running times of \(Z_{k}\) and of \(f_{k,0}\) are  of order \(k\). Since the length of the time-averaging period in  \(f_{k,0}\) is at least \(k/2\), it follows from Lemma~\ref{le:StdSum} that \begin{equation*}
\var(f_{k,0})\leq\frac{50\overline{\nu}(0)^{2}}{k}.
\end{equation*}Similarly, Lemma~\ref{le:specialPl} gives an upper bound on  \(E(Z_{k}^{2})\). Applying Proposition~\ref{pr:optimalq} with \(V_{1}=f_{k,0}\) and \(V_{2}=Z_{k}\) and replacing \(\var(f_{k,0})\) and \(E(Z_{k}^{2})\) with their upper bounds yields \eqref{eq:defq}, up to a multiplicative factor. \qed    
\section{Proof of Theorem~\ref{th:main}}        
It follows from Lemma~\ref{le:specialPl}  and the definition of \(\hat f_{k}\) that \(E(\hat f_{k})=E(f_{k,0})+E(Z_{k})=\mu\). Also, by \eqref{eq:varV},\begin{equation*}
\var(\hat f_{k})\le\var(f_{k,0}) +q^{-1}E(Z_{k}^{2}).
\end{equation*}
By Lemma~\ref{le:boundVarfk0}, \begin{displaymath}
\var(f_{k,0})\leq\frac{796({\overline\nu}(0))^{2}\beta(k)}{k}+\var(f_{k}),
\end{displaymath}
where
\begin{displaymath}
\beta(k):=\max\left(q,\sqrt{\frac{b'(k)-b(k)}{k}}\right).
\end{displaymath}Furthermore,
by Lemma~\ref{le:specialPl}, \begin{eqnarray*}
q^{-1}E(Z_{k}^{2})&\le&\frac{20(\overline{\nu}(\lfloor{b'(k)}/{2}\rfloor))^{2}}{qk}
\\&=&\frac{20( \overline{\nu}(0))^{2} q}{k}\\
&\le&\frac{20( \overline{\nu}(0))^{2}\beta(k)}{k}.
\end{eqnarray*}
Thus, \begin{equation*}
\var(\hat f_{k})\leq\var(f_{k})+ \frac{816({\overline\nu}(0))^{2}\beta(k)}{k}.
\end{equation*}
Lemma~\ref{le:specialPl} shows that  \(\hat T_{k}\le k+9qk\leq k(1+9\beta(k))\).  Hence,
 \begin{eqnarray*}
 \hat T_{k}\var(\hat f_{k})
 &\leq& (k\var(f_{k})+816({\overline\nu}(0))^{2}\beta(k))(1+9\beta(k))\\
&=& k\var(f_{k})+9k\var(f_{k})\beta(k)+816({\overline\nu}(0))^{2}\beta(k)+7344({\overline\nu}(0))^{2}(\beta(k))^{2}.
  \end{eqnarray*} 
As, by \eqref{eq:StdFnBound}, \(k\var(f_{k})\le50({\overline\nu}(0))^{2}\), and  \(\beta(k)\leq1\), this implies \eqref{eq:timeVarianceUBound} after some calculations.\qed
\section{Proof of Lemma~\ref{le:ObliviousPl}}
Denote by \(M\) the right-hand side of \eqref{eq:ZNSecondMomentGeneralCase}. By \eqref{eq:obliviouspl} and the monotonicity of \(\theta\), we have \(p_{l}\ge2^{-l-1}/\theta(l)\) for \(l\geq0\). As \(\theta(0)=\theta(1)=1\), it follows that \begin{equation*}
M\le12(\overline{\nu}(\lfloor{b'(k)}/{2}\rfloor))^{2}+16\sum^{\infty}_{l=2}\theta(l)(\overline{\nu}(k2^{l-2})-\overline{\nu}(k2^{l-1}))^{2}.
\end{equation*}
For \(l\geq2\), \begin{eqnarray*}\sqrt{\theta(l)}(\overline{\nu}(k2^{l-2})-\overline{\nu}(k2^{l-1}))&=&\sqrt{\theta(l)}\sum^{k2^{l-1}-1}_{i=k2^{l-2}}\sqrt{\frac{\nu(i)}{i+1}}\\
&\le&\sum^{k2^{l-1}-1}_{i=k2^{l-2}}\sqrt{\frac{\nu(i)\theta(\log_{2}(4i+1))}{i+1}}\\&=& \overline{\nu}_{\theta}(k2^{l-2})-\overline{\nu}_{\theta}(k2^{l-1}),\end{eqnarray*}
where the second equation follows from the monotonicity of \(\theta\). Hence, \begin{eqnarray*}\sum^{\infty}_{l=2}\theta(l)(\overline{\nu}(k2^{l-2})-\overline{\nu}(k2^{l-1}))^{2}&\le&\sum^{\infty}_{l=2}(\overline{\nu}_{\theta}(k2^{l-2})-\overline{\nu}_{\theta}(k2^{l-1}))^{2}\\
&\le&\overline{\nu}_{\theta}(k)\sum^{\infty}_{l=2}(\overline{\nu}_{\theta}(k2^{l-2})-\overline{\nu}_{\theta}(k2^{l-1}))\\&=&(\overline{\nu}_{\theta}(k))^{2}\\&\le&(\overline{\nu}_{\theta}(\lfloor{b'(k)}/{2}\rfloor))^{2},
\end{eqnarray*}
where the second and the last equation follow from the fact that  \(\overline{\nu}_{\theta}(i)\) is a decreasing function of \(i\). Thus \(M\le28(\overline{\nu}_{\theta}(\lfloor{b'(k)}/{2}\rfloor))^{2}\). Together with Lemma~\ref{le:generalPl}, this implies \eqref{eq:ZNSecondMomentOblivious} and that  \(E(f_{k,0}+Z_{k})=\mu\).  As \(p_{l}\le2^{-l}/\theta(l)\),  Lemma~\ref{le:generalPl} implies the desired bound on \(T_{k}\).   
\qed
\section{Proof of Theorem~\ref{th:maintheta}}
It follows from Lemma~\ref{le:ObliviousPl}  and the definition of \(\hat f_{k}\) that \(E(\hat f_{k})=E(f_{k,0})+E(Z_{k})=\mu\). Also, by definition of \(\hat f_{k}\) and \eqref{eq:varV},\begin{equation*}
\var(\hat f_{k})\le\var(f_{k,0}) +q^{-1}E(Z_{k}^{2}).
\end{equation*}
Lemma~\ref{le:boundVarfk0} implies that \begin{displaymath}
\var(f_{k,0})\leq\frac{796(\overline{\nu}(0))^{2}}{k}\sqrt{\frac{b'(k)-b(k)}{k}}+\var(f_{k}).
\end{displaymath}
Together with \eqref{eq:ZNSecondMomentOblivious}, this implies \eqref{eq:varfHatkOblivious}.
The desired bound on    \(\hat T_{k}\) follows from the bound on \(T_{k}\) in Lemma~\ref{le:ObliviousPl}.\qed
\section{Proof of Theorem~\ref{th:ObliviousPolynomialZeta} }
A standard calculation shows that \begin{displaymath}
\sum^{\infty}_{l=2}\frac{1}{\theta(l)}\leq\frac{1}{\delta-1},
\end{displaymath}
which implies  \eqref{eq:infiniteSumAssumptionZeta}. Also,   \eqref{eq:infiniteSumAssumptionZetaVu} follows from Assumption A2. Thus, Assumption A3 holds. By the inequality \(\delta\leq2\),
Theorem~\ref{th:maintheta} implies \eqref{eq:A2T^kbound}. Set \(\omega(i)=\nu(i)\theta(\log_{2}(4i+1))\) for \(i\geq0\).  For \(j\geq0\), we have \(\overline\omega(j)=\overline{\nu}_{\theta}(j)\). For \(i\geq0\),\begin{eqnarray*}\omega(i)&\leq&ce^{-\xi i}\max(1,\log_{2}(4i+1))^{\delta}\\
&\leq& ce^{-\xi i}(\log_{2}(4i+2))^{\delta}\\
&\leq& ce^{-\xi i}(2\log_{2}(i+2))^{\delta}\\
&\leq& \frac{4c}{\ln^{2}(2)} \ln^{\delta}(i+2)e^{-\xi i},
\end{eqnarray*} 
where the third equation follows from the inequality \(4i+2\leq(i+2)^{2}\), and the  last equation follows from the inequality \(\delta\leq2\). Using the inequality \(\delta\leq2\) once again,  \eqref{eq:S0upperBoundGen} implies that 
\begin{eqnarray*}
\overline{\nu}_{\theta}(0)
&\leq&\frac{2}{\ln(2)} \sqrt{\frac{ec}{\xi}}\left(2^{5/2}\Gamma\left(\frac{3}{2}\right)+2\sqrt{2\pi}\ln^{\delta/2}\left(\frac{2}{\xi}\right)\right)\\
&=&\frac{4\sqrt{2e\pi}}{\ln(2)} \sqrt{\frac{c}{\xi}}\left(1+\ln^{\delta/2}\left(\frac{2}{\xi}\right)\right)\\
&\le&\frac{8\sqrt{2e\pi}}{\ln(2)} \sqrt{\frac{c}{\xi}}\ln^{\delta/2}\left(\frac{3}{\xi}\right),
\end{eqnarray*}
where the second equation follows from the equality \(\Gamma(3/2)=\sqrt{\pi}/2\) and the last one from the fact that \(1\) and  \(\ln(2/\xi)\) are upper-bounded by  \(\ln(3/\xi)\). On the other hand,  \eqref{eq:SjUpperBoundGeom} and the inequality \(\delta\leq2\) show that, for \(j\geq 0\), \begin{displaymath}
\overline{\nu}_{\theta}(j)\leq {\frac{20\sqrt{c}}{e\ln(2)}}\frac{e^{-\xi j/2}}{\xi}.
\end{displaymath}
As \(\overline{\nu}_{\theta}(j)\le\overline{\nu}_{\theta}(0)\), it follows that
\begin{equation}\label{eq:nuBarZetaBound}
\overline{\nu}_{\theta}(j)\leq \min\left(\frac{8\sqrt{2e\pi}}{\ln(2)} \sqrt{\frac{c}{\xi}}\ln^{\delta/2}\left(\frac{3}{\xi}\right),{\frac{20\sqrt{c}}{e\ln(2)}}\frac{e^{-\xi j/2}}{\xi}\right).
\end{equation} 
Combining \eqref{eq:deltaBound}, \eqref{eq:nuBarZetaBound} and  \eqref{eq:varfHatkOblivious} yields \eqref{eq:kVarHatfkBoundExpoNu}.

Assume now that  \(\nu(i)\le c/(i+1)\) for \(i\geq0\). A calculation similar to the one above shows that, for \(i\geq0\),\begin{equation*}\omega(i)\leq\frac{4c}{\ln^{2}(2)} \frac{\ln^{\delta}(i+2)}{i+1}.
\end{equation*} As \(\delta\leq2\), by Proposition~\ref{pr:exponentialGeomDecay},\begin{eqnarray*}
\overline{\nu}_{\theta}(0)&\leq&\frac{14\sqrt{c}}{\ln(2)}\ln^{\delta/2+1}\left(\frac{8}{\xi^{2}}\right)\\&\leq&{\frac{56\sqrt{c}}{\ln(2)}}\ln^{\delta/2+1}\left(\frac{3}{\xi}\right).
\end{eqnarray*}
Hence, for \(j\geq0\), 
\begin{equation}\label{eq:nuBarZetaBoundGeom}
\overline{\nu}_{\theta}(j)\leq \min\left({\frac{56\sqrt{c}}{\ln(2)}}\ln^{\delta/2+1}\left(\frac{3}{\xi}\right),{\frac{20\sqrt{c}}{e\ln(2)}}\frac{e^{-\xi j/2}}{\xi}\right).
\end{equation}
Combining \eqref{eq:deltaBoundGeom} and \eqref{eq:nuBarZetaBoundGeom} with  \eqref{eq:varfHatkOblivious} yields \eqref{eq:kVarHatfkBoundExpoGoemNu}. \qed
\section{Proof of Proposition~\ref{pr:implementationDetails}}
 We first prove the correctness of Algorithm~\ref{alg:Bias}. Denote by \(S_{K}[0],\dots,S_{K}[h]\)   the values of  \(S[0],\dots,S[h]\)  at the end of LR.  It can be shown by  induction that if \(X_{0}[0]= X_{0}\) then  \(X_{i}[0]\sim X_{i}\) for \(0\leq i\leq K\) and \begin{equation}\label{eq:Sout[0]}
S_{K}[0]\sim\sum ^{K-1}_{i=B}f(X_{i}).
\end{equation}  Hence, at the end of line 4 of Algorithm~\ref{alg:Bias},  \(X[0]\sim X_{k2^{N}}\).

Assume now that  \(h=1\) in Algorithm~\ref{alg:LR}, and that  \(X_{0}[0]\sim X_{m}\) for some non-negative integer \(m\), and that \(X_{0}[1]= X_{0}\).  Denote by \(V_{0},V_{1},\dots,V_{K-1}\)  the successive copies of \(U_{0}\) generated by Algorithm~\ref{alg:LR}. We assume that \(X_{0}[0],V_{0},\dots,V_{K-1}\) are independent. We show by induction on \(i\) that, for \(0\leq i\leq K\),
\begin{equation}\label{eq:recursionMarkovChainDist}
((X_{j}[0],X_{j}[1]),0\leq j\leq i)\sim((X_{j-m',m+m'},X_{j-m',m'}),0\leq j\leq i). 
\end{equation}
The base case holds since \(X_{-m',m+m'}\sim X_{m}\). Assume that \eqref{eq:recursionMarkovChainDist} holds for \(i\). Step 11 in Algorithm~\ref{alg:LR} shows that  \(X_{i+1}[0]= g(X_{i}[0],V_{i})\) and \(X_{i+1}[1]= g(X_{i}[1],V_{i})\). Similarly, \eqref{eq:recXim} shows that \(X_{i+1-m',m+m'}=g(X_{i-m',m+m'},U_{i-m'})\) and  \(X_{i+1-m',m'}=g(X_{i-m',m'},U_{i-m'})\).   Since \(V_{i}\sim U_{i-m'}\) and \(V_{i}\) is independent of \((X_{i}[0],X_{i}[1])\), and \(U_{i-m'}\) is independent of \((X_{i-m',m+m'},X_{i-m',m'}),\) and both sides of \eqref{eq:recursionMarkovChainDist} are Markov chains, this implies that  \eqref{eq:recursionMarkovChainDist} holds for \(i+1\).    Thus,\begin{eqnarray}\label{eq:Sout01}
(S_{K}[0],S_{K}[1])&=&(\sum ^{K-1}_{i=B}f(X_{i}[0]),\sum ^{K-1}_{i=B}X_{i}[1])\nonumber\\
&\sim&(\sum ^{K-1}_{i=B}f(X_{i-m',m+m'}),\sum ^{K-1}_{i=B}f(X_{i-m',m'})).
\end{eqnarray}   
Applying \eqref{eq:Sout01} with \(B=k(2^{N}-1)+b'(k)\),  \(m=K=k2^{N}\) and \(m'=k(2^{N}-1)\), and using \eqref{eq:fklDef}, shows after some simplifications  that, at the end of line 6 of Algorithm~\ref{alg:Bias}, \begin{equation*}
\frac{1}{k-b'(k)}(S[0],S[1])\sim(f_{k,N+1},f_{k,N}).
\end{equation*}
By \eqref{eq:ZkDef}, line 7 of Algorithm~\ref{alg:Bias} outputs a random variable with the same distribution as  \(-Z_{k}^{(b'(k))}\). 

We now prove the correctness of Algorithm~\ref{alg:RHAlk}. By~\eqref{eq:Sout[0]}, at the end of line 3 of Algorithm~\ref{alg:RHAlk}, \begin{equation*}
S[0]\sim\sum^{k-1}_{i=b'(k)}f(X_{i}),
\end{equation*} and so line 4 is consistent with \eqref{eq:fk0def}.
  As line 7 is executed with probability \(1-q\)
and line 9 is executed with probability \(q\), \eqref{eq:defHatfk} shows that Algorithm~\ref{alg:RHAlk} outputs a random variable with the same distribution as \(\hat f_{k}\).
\qed
\section{Proof of Proposition~\ref{pr:GG1Binary}}
We have \(X_{n}=\max_{0\leq j\leq n}[S_{n}-S_{j}]\) for  \(n\geq0\),   where  \(S_{n}:=\sum^{n-1}_{k=0}U_{k}\), with \(S_{0}:=0\)   \citep[\S I, Eq.~(1.4)]{asmussenGlynn2007}). For \(i,m\ge0\), using \eqref{eq:recXim}, it can be shown by induction  on \(i\) that \(X_{i+m,-m}=\max_{m\leq j\leq i+m}[S_{i+m}-S_{j}]\).  Hence \begin{equation}\label{eq:XiplusmGG1Bin}
X_{i+m}=\max(X_{i+m,-m},\max_{0\leq j\leq m-1}[S_{i+m}-S_{j}]).
\end{equation} By applying Proposition~\ref{pr:shiftedPair} with \(m'=0\) and \(n=i\),  it follows that \begin{eqnarray*}E((f(X_{i,m})-f(X_{i}))^{2})&=&
E((f(X_{i+m})-f(X_{i+m,-m}))^{2})\\&\leq&\Pr (X_{i+m}\neq X_{i+m,-m})\\
&\leq&  \sum^{m-1}_{j=0}\Pr (S_{i+m}-S_{j}>0)
\\&=&\sum^{m-1}_{j=0}\Pr (S_{i+m-j}>0)
\\&=&\sum^{m}_{j=1}\Pr (S_{i+j}>0),
\end{eqnarray*}
where the third equation follows from \eqref{eq:XiplusmGG1Bin}, and  the fourth equation follows by noting that  \(S_{i+m}-S_{j}\sim S_{i+m-j}\). Set \(\mu'=E(U_{0})\) and \(\mu ''=E({U'_{0}}^{6})\), where \(U'_{i}=U_{i}-\mu'\) for \(i\geq0\). For \(n\geq0\), let   \(S'_{n}:=\sum^{n-1}_{j=0}U'_{j}\). Since \(U'_{n}\) and \(S'_{n}\) are centered and independent,  the binomial theorem shows that, for \(n\geq0\),\begin{eqnarray*}E({S'_{n+1}}^{6})
&=&E({S'_{n}}^{6})+15E({U'_{n}}^{2})E({S'_{n}}^{4})+15E({U'_{n}}^{4})E({S'_{n}}^{2})+E({U'_{n}}^{6})\\
&\le&E({S'_{n}}^{6})+15(\mu '')^{1/3}(E({S'_{n}}^{6}))^{2/3}+15(\mu '')^{2/3}(E({S'_{n}}^{6}))^{1/3}+\mu '',
\end{eqnarray*} where the second equation follows from Jensen's inequality \((E(V))^{\alpha}\leq E(V^{\alpha})\) for \(\alpha\geq1\) and non-negative random variable \(V\) with finite \(\alpha\)-moment. If follows by induction on \(n\) that \(E({S'_{n}}^{6})\leq125n^{3}\mu''\).
Hence\begin{eqnarray*}
\Pr(S_{n}>0)&=&\Pr(S'_{n}>-\mu'n)\\
&\le&E(\frac{{S'_{n}}^{6}}{(\mu'n)^{6}})\\
&\le&125\frac{\mu''}{(\mu')^{6}}n^{-3},
\end{eqnarray*}where the second equation follows from Markov's inequality and the fact that \(\mu'<0\). Hence,\begin{eqnarray*}
E((f(X_{i,m})-f(X_{i}))^{2})
&\le&125\frac{\mu''}{(\mu')^{6}}\sum^{\infty}_{j=1}(i+j)^{-3}\\
&\le&125\frac{\mu''}{(\mu')^{6}}\left((i+1)^{-3}+\int^{\infty}_{i+1}x^{-3}dx\right)\\
&\leq&\frac{375\mu''}{2(\mu')^{6}}(i+1)^{-2}.
\end{eqnarray*}\qed \section{Proof of Theorem~\ref{th:Gaussian}}
We first prove the exponential bound by showing that Proposition~\ref{pr:contractive1} holds with  \(\rho(x,x')=||\sqrt{V^{-1}}(x-x')||\) for \((x,x')\in F\times F\). Fix non-negative integers \(i\) and \(m\) with \(i+m\ge0\). 
It follows from \cite[Lemma 1]{kahaleGaussian2019} that, conditional on \(e_{-m},\dots,e_{i-1}\),   \(\begin{pmatrix}X_{i} \\
X_{i,m} \\
\end{pmatrix}\)  is a centered Gaussian vector with       \(\cov(X_{i})\le V\) and
  \(\cov(X_{i,m})\le V\). Consequently, by~\eqref{eq:defkappaLipsG},
\begin{equation*}
E((f(X_{i,m})-f(X_{i}))^{2}|e_{-m},\dots,e_{i-1})\leq \hat\kappa^{2}( E(||X_{i,m}-X_{i}||^{2}|e_{-m},\dots,e_{i-1}))^{\hat\gamma}.
\end{equation*}Taking expectations and using the tower law and Jensen's inequality implies that
\begin{equation}\label{eq:GaussianLipschitz}
E((f(X_{i,m})-f(X_{i}))^{2})\leq \hat\kappa^{2}( E(||X_{i,m}-X_{i}||^{2}))^{\hat\gamma}.
\end{equation}As \(\rho^{2}(x,x')=(x-x')^{T}{V^{-1}}(x-x')\), we have \(\rho^{2}(x,x')\ge \lambda_{\max}^{-1}||x-x'||^{2}\). Hence
\begin{equation*}
E((f(X_{i,m})-f(X_{i}))^{2})\leq \hat\kappa^{2}{\lambda_{\max}}^{\hat\gamma}( E(\rho^{2}(X_{i,m},X_{i})))^{\hat\gamma}.
\end{equation*}  
Thus, \eqref{eq:defkappaLips} holds with \(\kappa=\hat \kappa{\lambda_{\max}}^{\hat \gamma /2}\) and \(\gamma=\hat \gamma\). 

For \(m\ge0\), we have \(\rho(X_{0},X_{m})=\,||\sqrt{V^{-1}}X_{m}||\). It follows from   \cite[Lemma 2]{kahaleGaussian2019} that \(E(X_{m})=0\) and \(\cov(\sqrt{V^{-1}}X_{m})\leq I\). As the variance of each entry of \(\sqrt{V^{-1}}X_{m}\) is at most \(1\) and its expectation is \(0\), we have \(E(||\sqrt{V^{-1}}X_{m}||^{2})\leq d\). Consequently,  \eqref{eq:distanceFinite} holds with \(\kappa'=d\). 

Furthermore, for \(x,x'\in F\), we have
\begin{equation*}
g(x,g_{0},e_{0})-g(x',g_{0},e_{0})=(I-Ve_{0}e_{0}^{T})(x-x').
\end{equation*}Thus
\begin{equation*}
\rho(g(x,g_{0},e_{0}),g(x',g_{0},e_{0}))=||Py||,
\end{equation*}
where \(P=I-\sqrt{V}e_{0}e_{0}^{T}\sqrt{V}\) and \(y=\sqrt{V^{-1}}(x-x')\). A standard calculation (e.g., \citet{kahaleGaussian2019}) shows that \(P^{2}=P\) and \(E(P)=I-d^{-1}V\). Hence,\begin{eqnarray*}E(\rho^{2}(g(x,g_{0},e_{0}),g(x',g_{0},e_{0})))&=&E(y^{T}Py)\\
&=&y^{T}(I-d^{-1}V)y\\
&\le&(1-\lambda_{\min}/d)||y||^{2},
\end{eqnarray*}
 where the last equation follows from the fact that the largest eigenvalue of \(I-d^{-1}V\) is \(1-\lambda_{\min}/d\). As \(\rho(x,x')=||y||\), \eqref{eq:distanceContractive} holds for \(\eta=1-\lambda_{\min}/d\).
By Proposition~\ref{pr:contractive1}, \begin{equation*}
E(||f(X_{i,m})-f(X_{i})||^{2})\leq \hat\kappa^{2}(\lambda_{\max}d)^{\hat\gamma}(1-\lambda_{\min}/d)^{\hat\gamma i}.
\end{equation*}
We now prove the geometric bound. For \(n\geq0\), set  \(P_{n}:=I-\sqrt{V}e_{n}{e_{n}}^{T}\sqrt{V}\), and \(M_{n}:=P_{{n-1}}P_{{n-2}}\cdots P_{{0}}\), with \(M_{0}:=I\).  
 By~\eqref{eq:recXim}, for \(i,m\ge0\), \begin{equation}\label{eq:recXimGauss}
X_{i+1,m}=X_{i,m}+(g_{i}-e_{i}^{T}X_{i,m})(Ve_{i}).
\end{equation}As \eqref{eq:recXimGauss} also holds for \(m=0\), it follows that 
\begin{equation*}
X_{i+1,m}-X_{i+1}=(I-Ve_{i}e_{i}^{T})(X_{i,m}-X_{i}).
\end{equation*}
Consequently, it can be shown by induction on \(i\) that \(X_{i,m}-X_{i}=\sqrt{V}M_{i}\sqrt{V^{-1}}X_{0,m}\)  for \(i,m\ge0\). Hence 
\begin{eqnarray*} E(||X_{i,m}-X_{i}||^{2}) &=&E(||\sqrt{V}M_{i}\sqrt{V^{-1}}X_{0,m}||^{2})\\
&=&E(\tr(\sqrt{V}M_{i}\sqrt{V^{-1}}X_{0,m}{X_{0,m}}^{T}\sqrt{V^{-1}}{M_{i}}^{T}\sqrt{V}))\\
&=&E(\tr(\sqrt{V}M_{i}\sqrt{V^{-1}}E(X_{0,m}{X_{0,m}})\sqrt{V^{-1}}{M_{i}}^{T}\sqrt{V})),
\end{eqnarray*}
The second equation follows from the equality \(||Z||^{2}=\tr(ZZ^{T})\),  whereas the last equation follows from the independence of \(M_{i}\) and \(X_{0,m}\). On the other hand, \begin{eqnarray*}E(X_{0,m}{X_{0,m}}^{T})&=&E(X_m{X_{m}}^{T})\\&\le&V,\end{eqnarray*}
where the last equation follows from \cite[Lemma 2]{kahaleGaussian2019}. Consequently,
\begin{eqnarray*}
E(||X_{i,m}-X_{i}||^{2})
&\le&E(\tr(\sqrt{V}M_{i}{M_{i}}^{T}\sqrt{V}))\\
&=&E(\tr({M_{i}}^{T}VM_{i}))\\
&\le&\frac{d^{2}}{i+1},
\end{eqnarray*}
where the second equation follows from the equality \(\tr(AB)=\tr(BA)\), and the last equation follows from \cite[Theorem 1]{kahaleGaussian2019}.  Applying \eqref{eq:GaussianLipschitz} concludes the proof. 
\qed
\bibliography{poly}
\end{document}